\documentclass[11pt]{amsart}
\usepackage{amsmath,amssymb,mathrsfs}
\usepackage{amsrefs}
\usepackage[arrow,matrix]{xy}
\usepackage{enumerate}

\newcommand\cO{\mathcal O}
\newcommand\cL{\mathcal L}
\newcommand\cD{\mathcal D}
\newcommand\cC{\mathcal C}
\newcommand\cZ{\mathcal Z}

\newcommand\cX{\mathcal X}
\newcommand\cY{\mathcal Y}

\newcommand\cU{\mathcal U}
\newcommand\cT{\mathcal T}

\newcommand\cS{\mathcal S}

\newcommand\bA{\mathbb A}
\newcommand\bP{\mathbb P}
\newcommand\bZ{\mathbb Z}
\newcommand\bQ{\mathbb Q}
\newcommand\bC{\mathbb C}
\newcommand\bG{\mathbb G}
\newcommand\bF{\mathbb F}
\newcommand\bN{\mathbb N}
\newcommand\bK{\mathbb K}
\newcommand\bS{\mathbb S}
\newcommand\bR{\mathbb R}

\newcommand\fS{\mathfrak S}

\newcommand\fp{\mathfrak p}
\newcommand\fq{\mathfrak q}

\DeclareMathOperator{\Spec}{Spec}
\DeclareMathOperator{\Proj}{Proj}
\DeclareMathOperator{\charac}{char}
\DeclareMathOperator{\Pic}{Pic}

\DeclareMathOperator{\lcm}{lcm}
\DeclareMathOperator{\Gal}{Gal}

\DeclareMathOperator{\Frac}{Frac}

\DeclareMathOperator{\Ker}{Ker}
\DeclareMathOperator{\Coker}{Coker}
\DeclareMathOperator{\Image}{Im}
\DeclareMathOperator{\NS}{NS}
\DeclareMathOperator{\pr}{pr}

\DeclareMathOperator{\Gr}{Gr}
\DeclareMathOperator{\End}{End}
\DeclareMathOperator{\rank}{rank}
\DeclareMathOperator{\Res}{Res}
\DeclareMathOperator{\SO}{SO}
\DeclareMathOperator{\U}{U}
\DeclareMathOperator{\Sh}{Sh}
\DeclareMathOperator{\disc}{disc}
\DeclareMathOperator{\GL}{GL}
\DeclareMathOperator{\GSpin}{GSpin}
\DeclareMathOperator{\GSp}{GSp}

\DeclareMathOperator{\Hdg}{Hdg}
\DeclareMathOperator{\height}{ht}

\DeclareMathOperator{\Hom}{Hom}
\DeclareMathOperator{\sHom}{\mathcal{H}\mathit{om}}
\DeclareMathOperator{\norm}{Nm}
\DeclareMathOperator{\trace}{tr}

\newcommand\et{\mathrm{\acute et}}
\newcommand\crys{\mathrm{crys}}
\newcommand\logcrys{\mathrm{logcrys}}
\newcommand\st{\mathrm{st}}
\newcommand\pst{\mathrm{pst}}

\newcommand\dR{\mathrm{dR}}

\newcommand\sm{\mathrm{sm}}
\newcommand\sep{\mathrm{sep}}
\newcommand\Cst{\ensuremath{C_\st}}
\newcommand\ur{\mathrm{ur}}
\newcommand\abel{\mathrm{ab}}
\newcommand\tame{\mathrm{tame}}

\newcommand\rationalto{\to}

\newcommand\isomto{\stackrel{\sim}{\to}}

 \theoremstyle{plain}
 \newtheorem{thm}{Theorem}[section]
 \newtheorem{lem}[thm]{Lemma}
 \newtheorem{prop}[thm]{Proposition}
 \newtheorem{cor}[thm]{Corollary}

 \theoremstyle{definition}
 \newtheorem{rem}[thm]{Remark}
 \newtheorem{defn}[thm]{Definition}

\begin{document}

\title{Good reduction criterion for K3 surfaces}
\author{Yuya Matsumoto}
\address{Graduate School of Mathematics, Nagoya University}
\email{matsumoto.yuya@math.nagoya-u.ac.jp}
\date{2017/01/05.
The final publication is available at Springer via \url{http://dx.doi.org/10.1007/s00209-014-1365-8}
(Math.\ Z. \textbf{279} (2015), no. 1--2, 241--266).
}

\begin{abstract}
We prove a N\'eron--Ogg--Shafarevich type criterion for good reduction of K3 surfaces,
which states that a K3 surface over a complete discrete valuation field has potential good reduction 
if its $l$-adic cohomology group is unramified.
We also prove a $p$-adic version of the criterion.
(These are analogues of the criteria for good reduction of abelian varieties.)
The model of the surface will be in general not a scheme but an algebraic space.
As a corollary of the criterion we obtain the surjectivity of the period map of K3 surfaces in positive characteristic.
\end{abstract}

\maketitle

\section{Introduction}

Let $K$ be a complete discrete valuation field with perfect residue field of characteristic $p$.
The N\'eron--Ogg--Shafarevich criterion states that
an abelian variety $A$ has good reduction if and only if
 $H^1_\et(A_{\overline K}, \bQ_l)$ is an unramified representation\footnote
{A representation of the absolute Galois group of a discrete valuation field is said to be unramified if the inertia subgroup acts trivially.}
 (or equivalently, the $l$-adic Tate module of $A$ is unramified) 
for some prime $l \neq p$ (then it is so for all $l \neq p$).
A $p$-adic counterpart of this result is that
$A$ has good reduction if and only if $H^1_\et(A_{\overline K}, \bQ_p)$ is a crystalline representation.

Such criteria do not hold for general varieties 
(if a variety has good reduction then its cohomology groups are unramified/crystalline, 
but the converse is not true).

In this paper we prove criteria
for K3 surfaces
(in both $l$-adic and $p$-adic settings)
similar to those for abelian varieties.
We allow \emph{algebraic spaces} as models 
and consider \emph{potential} good reduction 
(that is, we allow finite extension of the base field). 
More precisely, our main theorem is the following.

\begin{thm} \label{thm:maintheorem}
Let $X$ be a K3 surface over $K$
which admits an ample line bundle $L$ satisfying $p > L^2 + 4$.
Assume that one of the following holds:
\begin{enumerate}[\rm (a)]
\item For some prime $l\neq p$, $H^2_\et(X_{\overline K}, \bQ_l)$ is unramified.
\item ($K$ is of characteristic $0$ and) $H^2_\et(X_{\overline K}, \bQ_p)$ is crystalline.
\end{enumerate}
Then 
$X$ has potential good reduction with an algebraic space model, 
that is, for some finite extension $K'/K$, 
there exists an algebraic space smooth proper over $\cO_{K'}$ with generic fiber isomorphic to $X_{K'}$.
\end{thm}

\begin{rem} \label{rem:maintheorem}
(1)
We cannot replace ``algebraic space'' in the statement of the theorem by ``scheme'';
we present counterexamples in Section \ref{subsec:example-algsp}.
Hence it is, in contrast to the case of abelian varieties,
 somewhat essential to allow algebraic spaces (not only schemes)
when considering reduction of K3 surfaces.

(2)
We do not know whether the field extension is necessary. 
\footnote{This problem is considered further in my subsequent paper with Christian Liedtke \cite{Liedtke--Matsumoto}.}
Under an additional assumption we give a bound for the extension degree (Corollary \ref{cor:mainthm-degree2}).
See also the example in Section \ref{subsec:example-unram}.

(3) 
If one prefers to stay in the category of schemes,
we have the following:
If $X$ is as in the theorem, 
then for some finite extension $K'/K$, 
there exists a proper flat scheme $\cX$ over $\cO_{K'}$ with generic fiber isomorphic to $X_{K'}$ and
special fiber having at worst rational double point singularities.

(4)
The condition on the degree $L^2$ of the line bundle is satisfied for example in the following cases:
(i) $p \geq 7$ and $X$ is a double cover of $\bP^2$ ramifying over a sextic; or
(ii) $p \geq 11$ and $X$ is a quartic surface in $\bP^3$.
On the other hand, for each $p$ there exist (infinitely many) K3 surfaces not satisfying the condition, 
and for $2 \leq p \leq 5$ the condition is never satisfied.

(5)
If we know in advance that $X$ (or some surface birationally equivalent to $X$) has potential semistable reduction,
the condition $p > L^2 + 4$ can be replaced by the (weaker) condition $p \geq 5$, 
since the inequality $p > L^2 + 4$ is used only in the first step of the proof (see the outline below).
Potential semistable reduction of general surfaces in mixed characteristic is an open problem.
\end{rem}

We have two applications of the main theorem.
First, we can prove the surjectivity of the period map of K3 surfaces
in positive characteristic (Theorem \ref{thm:surjectivity}).
Second, we can show that K3 surfaces having complex multiplication (see Definition \ref{def:CMK3})
have potential good reduction (Theorem \ref{thm:CMpotgoodred}).

\medskip

We review known results concerning this kind of criterion.

Serre--Tate \cite{Serre--Tate}*{Theorem 1} proved the criterion for abelian varieties
and gave the name ``N\'eron--Ogg--Shafarevich criterion''
after related works of N\'eron \cite{Neron:modeles}, Ogg \cite{Ogg:elliptic}, and Shafarevich.
The $p$-adic version is obtained by Coleman--Iovita \cite{Coleman--Iovita}*{Theorem II.4.7} and Breuil \cite{Breuil:p-divisibles}*{Corollaire 1.6}.
These criteria fail for general varieties, even for curves of genus $\geq 2$
(but there are results of Oda \cite{Oda:fundamentalgroup}*{Theorem 3.2} and Andreatta--Iovita--Kim \cite{Andreatta--Iovita--Kim}*{Theorem 1.6} 
relating good reduction of curves with $l$-adic and $p$-adic fundamental groups respectively).
Kulikov \cite{Kulikov:degeneration} essentially showed the (potential) criterion for complex K3 surfaces 
in the category of complex manifolds (not necessarily schemes).
In mixed characteristic, 
Ito \cite{Ito:Kummer}*{Corollary 4.3}\footnote{
This paper of Ito is unpublished, but it is included in \cite{Matsumoto:SIP}*{Theorem 1.18} as an appendix.}
 and myself \cite{Matsumoto:SIP}*{Theorem 0.1} proved analogues of the N\'eron--Ogg--Shafarevich criterion 
for some special kinds of K3 surfaces
using the geometry of the surfaces (which are closely related to abelian surfaces).

\medskip

The outline of the proof of Theorem \ref{thm:maintheorem} (given in Section \ref{sec:proof}) is as follows.

We follow a method of Maulik \cite{Maulik:supersingular} of studying reduction of K3 surfaces.
Using the line bundle $L$ and results of Saito \cite{Saito:logsmooth} and Kawamata \cite{Kawamata:mixed3fold},
we obtain a model of $X$ with log terminal singularities which are well described.
Using a result of Artin \cite{Artin:brieskorn},
we can resolve some of the singularities of that model 
in the category of algebraic spaces
and we obtain a strictly semistable model (which is an algebraic space).
The special fiber of that space is then an SNC (simple normal crossing) log K3 surface, 
which is classified by Nakkajima \cite{Nakkajima:logk3}
(in a parallel way to the characteristic $0$ case \cite{Kulikov:degeneration}).
Using the unramified/crystalline assumption and a comparison result between the cohomology groups of the generic and the special fiber,
we conclude that the special fiber is actually smooth.

In the last step we need some comparison results, 
which are not well-known since the model is not a scheme in general.
In the $p$-adic case (Section \ref{sec:p-adic}), 
we use Olsson's results \cite{Olsson:crystalline} on the Hyodo--Kato isomorphisms for algebraic spaces.
In the $l$-adic case (Section \ref{sec:l-adic}), 
we generalize the (Steenbrink--Rapoport--Zink) weight spectral sequence 
to the algebraic space case.

\subsection*{Acknowledgments}
The author expresses his sincere gratitude to his advisor 
Atsushi Shiho for supporting him in many ways.
The author also thanks 
Takuma Hayashi, Tetsushi Ito, Teruhisa Koshikawa, Keerthi Madapusi Pera, Yukiyoshi Nakkajima, Takeshi Saito, Naoya Umezaki, and Kohei Yahiro
for giving him helpful comments.
This work was supported by Grant-in-Aid for JSPS Fellows Grant Number 12J08397.

\section{Comparison theorems for semistable algebraic spaces} \label{sec:ss}

In this section we prove comparison theorems of 
($l$-adic and $p$-adic) cohomology groups of fibers of a semistable algebraic space,
which we will use in the proof of the main theorem.
For general properties of algebraic spaces, see \cite{Knutson:algebraicspaces}.

\subsection{Statement of the comparison theorems}

Let $\cO_K$ be a complete discrete valuation ring with perfect residue field of characteristic $p > 0$.
Denote by $G_K$ the absolute Galois group of $K$.
We first introduce the notion of semistable algebraic spaces over $\cO_K$.

Recall that an algebraic space is said to be irreducible if 
the intersection of any two nonempty open subspaces (images of open immersions) is nonempty,
and that an irreducible component of an algebraic space is a maximal irreducible closed subspace.

\begin{defn}

An algebraic space $X$ over $\cO_K$ is said to be \emph{semistable} purely of dimension $n$ if
it is \'etale-locally isomorphic to $\cO_K[x_1, \ldots, x_{n+1}] \allowbreak / \allowbreak (x_1 \cdots x_r - \pi)$,
where $\pi$ is a uniformizer of $\cO_K$.
It is \emph{strictly semistable} if moreover 
each irreducible component of the special fiber is smooth.

\end{defn}
Although irreducibility is not an \'etale local property, 
an \'etale covering of a strictly semistable algebraic space is always strictly semistable.
If $X$ is a scheme, these definitions are of course equivalent to the usual ones.

We use the following notation for a strictly semistable algebraic space $X$ over $\Spec \cO_K$.
\begin{itemize}
\item
$X_{K}$ and $X_k$ are the generic and special fibers of $X$.

\item
$X_{\overline K}$ and $X_{\overline k}$ are the corresponding geometric fibers.

\item
$Z_h$ ($h = 1, \ldots, m$) are the irreducible components of $X_k$.
By assumption each $Z_h$ is smooth and 
hence all connected components of each $(Z_h)_{\overline k}$ are smooth (in particular irreducible).

\item
$X_k^{(p)}$ is
 the disjoint union of the
smooth $(n-p)$-dimensional (possibly empty) subspaces $Z_H = Z_{h_0} \cap \cdots \cap Z_{h_p}$
for sets $H = \{ h_0, \ldots, h_{p} \} \subset \{ 1, \ldots, m \}$ of cardinality\footnote{
This numbering is the same to that of Saito \cite{Saito:weightSS}.
Some authors (e.g. \cite{Rapoport--Zink:monodromie}, \cite{Mokrane:spectrale}, \cite{Illusie:autour}) write $-^{(p+1)}$ for this space.
} $p+1$.

\item
$X_{\overline k}^{(p)}$ is constructed similarly from the irreducible components of $X_{\overline k}$.
(Since each $(Z_h)_{\overline k}$ is smooth, this is naturally identified  with $X_k^{(p)} \times_k \overline k$).
\end{itemize}

We now state the comparison theorems. 
Proofs will be given in the following subsections.

\begin{prop} \label{prop:p-ss}
Assume $K$ is of characteristic $0$. 
Let $X$ be a proper strictly semistable algebraic space over $\cO_K$
whose fibers are $2$-dimensional schemes.
Let $W = W(k)$ and $K_0 = \Frac W$.

(1) We have a $p$-adic spectral sequence
\[
 E_1^{p,q} = \bigoplus_{i \geq \max\{0,-p\}} H^{q-2i}_\crys (X_k^{(p+2i)}/W)(-i)
  \Rightarrow H^{p+q}_\logcrys(X_k/W).
\]
Moreover this spectral sequence is compatible with the monodromy operator in the following sense:
There is an endomorphism $N$ on the Hyodo--Steenbrink complex $W_n A^{\bullet}$ defined by Mokrane \cite{Mokrane:spectrale}*{Section 3.13}
which induces a map 
\[
\xymatrix{ \displaystyle
 E_1^{p,q} = \bigoplus _{i \geq \max\{0,-p\}} H^{q-2i}_\crys (X_k^{(p+2i)}/W)(-i)
  \ar@{=>}[r]
  \ar[d]^{N} 
 & H^{p+q}_\logcrys(X_k/W)
  \ar[d]^{N} \\
  \displaystyle
 E_1^{p+2,q-2} = \bigoplus _{i-1 \geq \max\{0,-p-2\}} H^{q-2i}_\crys (X_k^{(p+2i)}/W)(-i+1)
  \ar@{=>}[r]
 & H^{p+q}_\logcrys(X_k/W)
}
\]
of spectral sequences.

(2) The spectral sequence degenerates at $E_2$ modulo torsion.

(3) The morphism $N$ induces the following isomorphisms on $E_2$ terms modulo torsion
\[
 N   \colon {E_2^{-1,3}}_\bQ \isomto E_2^{1,1}(-1)_\bQ, \quad
 N^2 \colon {E_2^{-2,4}}_\bQ \isomto E_2^{2,0}(-2)_\bQ.
\]

(4) 
Assume moreover that $X_k$ is liftable to a semistable scheme over $K_0$.
(This liftability assumption is satisfied, for example,
if $X_k$ is a projective SNC log K3 surface of Type II or III
\cite{Nakkajima:logk3}*{Corollary 6.9}.\footnote{
In the published version this ``Type II or III'' assumption is missed.
This does not affect the proof of the main theorem since if $X_k$ is Type I 
then there is nothing to prove.}) 
Then we have an isomorphism 
\[
 H^2_\logcrys(X_k/W) \otimes_{W} K \cong (D_\pst(H^2_\et(X_{\overline K}, \bQ_p)) \otimes_{K_0^\ur} \overline K)^{G_K}
\]
compatible with the operator $N$.
In particular, if $H^2_\et$ is crystalline, 
then the operator $N$ on the right hand side of the spectral sequence in (1) is zero modulo torsion.
\end{prop}

\begin{prop}  \label{prop:weightss}
Let $X$ be as in the previous proposition (with no assumption on $\charac K$).
Let $l \neq p$ be a prime.
Let $\Lambda$ be $\bZ/l^n\bZ$, $\bZ_l$, or $\bQ_l$.

(1) We have an $l$-adic spectral sequence
\[
 E_1^{p,q} = \bigoplus_{i \geq \max\{0,-p\}} H^{q-2i}_\et (X_{\overline k}^{(p+2i)}, \Lambda(-i))
  \Rightarrow H^{p+q}_\et(X_{\overline K}, \Lambda)
\]
compatible with the action of $G_K$.
Moreover this spectral sequence is compatible with the monodromy operator in the following sense:
Let $T$ be an element of the inertia group $I_K$ such that $t_l(T)$ is a generator of $\bZ_l(1)$
(where $t_l \colon I_K \to \bZ_l(1)$ is the canonical surjection). 
Then the endomorphism $N = T-1$ of the complex $R \psi \Lambda$ of nearby cycles (defined later) induces a map
\[
\xymatrix{ \displaystyle
 E_1^{p,q} = \bigoplus _{i \geq \max\{0,-p\}} H^{q-2i} (X_{\overline k}^{(p+2i)}, \Lambda(-i))
  \ar@{=>}[r]
  \ar[d]^{1 \otimes t_l(T)} 
 & H^{p+q}(X_{\overline K}, \Lambda) 
  \ar[d]^{N} \\
  \displaystyle
 E_1^{p+2,q-2} = \bigoplus _{i-1 \geq \max\{0,-p-2\}} H^{q-2i} (X_{\overline k}^{(p+2i)}, \Lambda(-i+1))
  \ar@{=>}[r]
 & H^{p+q}(X_{\overline K}, \Lambda) \\
}
\]
of spectral sequences\footnote{
  We fixed a small error in the corresponding formula in \cite{Saito:weightSS}*{Corollary 2.8}: 
  his $i-1 \geq \max\{0,-p\}$ should be $i-1 \geq \max\{0,-p-2\}$.
}.

(2) The spectral sequence degenerates at $E_2$ modulo torsion.

(3) Let $\Lambda = \bQ_l$. 
The morphism $N$ induces the following isomorphisms on $E_2$ terms 
\[
 N   \colon E_2^{-1,3} \isomto E_2^{1,1}(-1), \quad
 N^2 \colon E_2^{-2,4} \isomto E_2^{2,0}(-2).
\]
\end{prop}

\subsection{Proof of $p$-adic case} \label{sec:p-adic}

We prove Proposition \ref{prop:p-ss}.

Note that (1)--(3) are statements for the special fiber, which is a (log) scheme.

(1),(2)
These are proved (without restriction on the dimension) by Mokrane \cite{Mokrane:spectrale}*{Section 3.23 and Proposition 3.18} 
and Nakkajima \cite{Nakkajima:spectral}*{Theorem 3.6} respectively.

(3)
By \cite{Nakkajima:signs}*{Remark 6.8 (1)} (cf.~\cite{Nakkajima:monodromies}*{Theorem 8.3})
the weight monodromy conjecture is true for $H^2$ of $2$-dimensional log schemes.

(4)
If the algebraic space $\cX$ is a scheme,
then this follows from the isomorphism of \Cst.
We use results of Olsson \cite{Olsson:crystalline} to extend this argument to algebraic spaces
(actually he treats a more general case of tame Deligne--Mumford stacks).
A $(\phi,N,G_K)$-structure (in the sense of Olsson \cite{Olsson:crystalline}*{Definition 0.1.1}) on a $K$-vector space $T$ is 
a collection $(D, \phi, N, \{ \rho_\pi \})$ of
a finite dimensional $K_0^\ur$-vector space $D$ with 
a continuous semilinear $G_K$-action, 
a Frobenius-semilinear automorphism $\phi$ on $D$, 
a $K_0^\ur$-linear nilpotent endomorphism $N$ on $D$, 
and a family of isomorphisms $\rho_\pi \colon T \stackrel\sim\to (D \otimes_{K_0^\ur} \overline K)^{G_K}$ 
indexed by the uniformizers $\pi$ of $K$,
subject to certain compatibilities.

Since the generic fiber $X_K$ of our $\cX$ is a scheme,
$H^2_\dR(X_K/K)$ is equipped with a $(\phi,N,G_K)$-structure with $D = D_\pst(H^2_\et(X_{\overline K}, \bQ_p)) $
through the \Cst\ isomorphism of Tsuji (and de Jong's alteration).
Olsson defined another $(\phi,N,G_K)$-structure on $H^2_\dR(X_K/K)$ in the following way:
Let $\cS_H(\alpha)$ be the stack (over $\bZ[t]$) defined in \cite{Olsson:crystalline}*{Sections 6.1--6.2}
with parameters $\alpha = (\alpha_1, \ldots, \alpha_r) = (1, \ldots, 1)$ and $H = \fS_r$ for $r$ large enough,
and $\cS = \cS_H(\alpha)_{\cO_K,\pi}$ be its base change by $\bZ[t] \to \cO_K \colon t \mapsto \pi$.
Any semistable algebraic space over $\cO_K$ with dimension at most $r$
admits a canonical smooth morphism to $\cS$ (\cite{Olsson:crystalline}*{Section 6.2}).
Hence we have a smooth morphism $\cX \to \cS$.
Then Olsson defines (\cite{Olsson:crystalline}*{Sections 6.4--6.5}) 
a projective system $(H^2_\crys(X_k / \cS_{W_n}))_n$ equipped with $\phi$ and $N$,
which depends only on the special fiber $X_k$,
and he shows that $D = (\varprojlim_{n} H^2_\crys(X_k / \cS_{W_n})) \otimes_W K_0^{\ur}$
gives a $(\phi,N,G_K)$-structure on $H^2_\dR(X_K/K)$.
Here $\cS_{W_n} = \cS_H(\alpha)_{W_n,0}$ is defined in the same way.
He shows also that these two $(\phi,N,G_K)$-structures are isomorphic (\cite{Olsson:crystalline}*{Theorem 9.6.9}).
In particular, we have an isomorphism
\[
 (D_\pst(H^2_\et(X_{\overline K}, \bQ_p)) \otimes_{K_0^\ur} \overline K)^{G_K} 
 \isomto (\varprojlim_{n} H^2_\crys(X_k / \cS_{W_n})) \otimes_W K
\]
compatible with $N$.

On the other hand, by assumption we have a semistable scheme $Y$ over $\cO_{K_0}$ with special fiber isomorphic to $X_k$.
Then $H^2_\dR(Y_{K_0}/K_0)$ admits a $(\phi,N,G_{K_0})$-structure with 
$D = (\varprojlim_{n} H^2_\crys(Y_k / \cS_{W_n})) \otimes K_0^{\ur}$. 
There is also a $(\phi,N,G_{K_0})$-structure 
with $D = H^2_\logcrys(Y_k/W) \otimes_W K_0^\ur$ 
defined by Hyodo--Kato \cite{Hyodo--Kato}*{Section 3}.
He shows that these two $(\phi,N,G_{K_0})$-structures are isomorphic (\cite{Olsson:crystalline}*{Theorem 9.6.7}).
In particular, we have an isomorphism
\[
 (\varprojlim_{n} H^2_\crys(Y_k / \cS_{W_n})) \otimes_W K_0 
 \isomto H^2_\logcrys(Y_k/W) \otimes_W K_0
\]
compatible with $N$.

Combining these with the obvious isomorphisms
\begin{align*}
 (\varprojlim_{n} H^2_\crys(X_k / \cS_{W_n})) \otimes_W K
 & \cong (\varprojlim_{n} H^2_\crys(Y_k / \cS_{W_n})) \otimes_W K \quad \text{and} \\
 H^2_\logcrys(X_k/W) \otimes_W K
 & \cong H^2_\logcrys(Y_k/W) \otimes_W K ,
\end{align*}
we obtain the desired isomorphism.
(The compatibility condition of the $(\phi,N,G_K)$-structure implies that
this isomorphism is independent of the choice of $\pi$ used in the constructions.
We do not need this.)

\subsection{Proof of $l$-adic case} \label{sec:l-adic}

In this subsection we prove Proposition \ref{prop:weightss}.

We will follow Saito's construction \cite{Saito:weightSS}*{Sections 1--2}
to formulate the weight spectral sequence of $l$-adic cohomology groups for semistable algebraic spaces.
The point of his construction is to give a filtration of $R \psi \Lambda$ by (shifted) perverse sheaves (Lemma \ref{wss25}).

First, 
we will need the theory of \'etale sheaves, derived categories, six functors, and perverse sheaves 
on algebraic spaces.
This is developed by Laszlo--Olsson in
\cites{Laszlo--Olsson:derived1,Laszlo--Olsson:derived2,Laszlo--Olsson:perverse}.
Since algebraic spaces ``are \'etale-locally schemes'', 
most properties of algebraic spaces (or objects on algebraic spaces)
can be defined by taking (the pullback by) an \'etale covering by a scheme.
For example, a perverse sheaf on an algebraic space $Y$ of finite type over a field is defined to be 
an object of the derived category $D^b_c(Y, \Lambda)$ 
whose pullback by some (equivalently any) \'etale covering $U \rightarrow X$ by a scheme 
is a perverse sheaf on $U$.

Let $f \colon X \rightarrow S = \Spec \cO_K$ be a strictly semistable algebraic space purely of relative dimension $n$. 
Define immersions $i$, $\bar i$, $j$, $\bar j$
by the diagram
\[
\xymatrix{
X_{\overline k} \ar[r]^{\bar i} \ar[d] &
X_{S^{\ur}} \ar[d] &
X_{\overline K} \ar[l]_{\bar j} \ar[d] \\
X_k \ar[r]^{i} &
X &
X_K \ar[l]_{j} .
}
\]
Let $R \psi \Lambda = \bar i^* R \bar j_* \Lambda$
 be the complex of nearby cycles, 
 which we regard as an object of the derived category $D^+_c(X_{k}, \Lambda)$
with a continuous action of $G_K$.

We introduce some morphisms:
\begin{itemize}
\item 
$i_h \colon Z_h \hookrightarrow X$ and $j_h \colon X \setminus Z_h \hookrightarrow X$ are the immersions.

\item 
$a_p \colon X_k^{(p)} \to X_k$ is the natural map (induced by the immersions).

\item
$i^{(p)} = i \circ a_p \colon X_k^{(p)} \to X$.

\item 
$\theta \colon \Lambda \to i^* R^1 j_* \Lambda(1)$ is the map 
sending $1$ to the boundary $\partial [\pi] \in H^1(K, \Lambda(1))$ of a prime element $\pi$ of $K$ with respect to the Kummer sequence.

\item 
$\theta_h \colon \Lambda_{Z_h} \to i_h^* R^1 j_{h *} \Lambda(1)$
are defined similarly. 

\item 
$\theta' \colon a_{0 *} \Lambda \to i^* R^1 j_* \Lambda(1)$ is the direct sum of the $\theta_h$'s.
\end{itemize}
These maps are all defined without problems in the algebraic space case.

\begin{lem}
 [cf.\ \cite{Saito:weightSS}*{Proposition 1.2, Corollary 1.3}] \label{wss12}
(1)
The map $\theta' \colon a_{0 *} \Lambda \rightarrow i^* R^1 j_* \Lambda(1)$
is an isomorphism and induces isomorphisms
$\theta' \colon a_{p *} \Lambda \rightarrow i^* R^{p+1} j_* \Lambda(p+1)$
for $p \geq 0$ by cup-product.

(2) Let $p \geq 0$.
The canonical map $i^* R^p j_* \Lambda \rightarrow R^p \psi \Lambda$ is surjective.
The map $\theta \cup \colon i^* R^p j_* \Lambda \rightarrow i^* R^{p+1} j_* \Lambda(1)$
induces a map $\bar \theta \colon R^p \psi \Lambda \rightarrow i^* R^{p+1} j_* \Lambda(1)$.
The sequence
\[
 0 \to R^p \psi \Lambda \stackrel{\bar \theta}{\to} i^* R^{p+1} j_* \Lambda(1) \to R^{p+1} \psi \Lambda(1) \rightarrow 0
\]
is exact.

(3)
Let $\delta \colon \Lambda \to a_{0 *} \Lambda$ be the canonical map.
Then we have an isomorphism
\[
\xymatrix{
0 \ar[r] & 
\Lambda \ar[r]^\delta \ar@{=}[d]&
a_{0*} \Lambda \ar[r]^{\delta \wedge} \ar[d]^{\theta'} &
\cdots \ar[r]^{\delta \wedge} &
a_{n *} \Lambda \ar[r] \ar[d]^{\theta'} &
0 \\
0 \ar[r] & 
\Lambda \ar[r]^(0.32)\theta &
i^* R^1 j_* \Lambda(1) \ar[r]^(0.6){\theta \cup} &
\cdots \ar[r]^(0.3){\theta \cup} &
i^* R^{n+1} j_* \Lambda(n+1) \ar[r] &
0 \\
}
\]
of exact sequences.

(4)
For $p \geq 0$, 
we have an exact sequence
\[
 0 \to R^p \psi \Lambda 
 \stackrel{\bar \theta}{\to} i^* R^{p+1} j_* \Lambda (1)
 \stackrel{\theta \cup}{\to} \cdots
 \stackrel{\theta \cup}{\to} i^* R^{n+1} j_* \Lambda (n-p+1)
 \to 0.
\]
\end{lem}
\begin{proof}
(1), (2) Since the assertions are \'etale local, 
we can reduce to the scheme case \cite{Saito:weightSS}.

(3), (4) These follow from (1) and (2).
\end{proof}

\begin{lem}
 [cf.\ \cite{Saito:weightSS}*{Lemma 2.5}] \label{wss25}
(1)
The object $R \psi \Lambda$ of $D_c^b(X_{\overline k}, \Lambda)$ is a $(-n)$-shifted perverse sheaf.

(2)
The canonical filtration $F'_p R \psi \Lambda = \tau_{\leq p} R \psi \Lambda$ is a filtration by sub $(-n)$-shifted perverse sheaves.
Here $\tau_{\leq \bullet}$ is the truncation.

(3)
For an integer $p \geq 0$, 
the map $\bar \theta \colon R^p \psi \Lambda \rightarrow i^* R^{p+1} j_* \Lambda(1)$
induces an isomorphism
\[
\Gr_p^{F'} R \psi \Lambda
 = R^p \psi \Lambda[-p]
 \stackrel{\bar \theta}{\rightarrow}
 [
  i^* R^{p+1} j_* \Lambda(1)
  \stackrel{\theta \cup}{\rightarrow}
  \cdots
  \stackrel{\theta \cup}{\rightarrow}
  i^* R^{n+1} j_* \Lambda(n-p+1)
 ]
\]
where $i^* R^{n+1} j_* \Lambda(n-p+1)$ is put on degree $n$.
The truncation 
\[
 [
  i^* R^{p+q+1} j_* \Lambda(q+1)
  \stackrel{\theta \cup}{\rightarrow}
  \cdots
  \stackrel{\theta \cup}{\rightarrow}
  i^* R^{n+1} j_* \Lambda(n-p+1)
 ]
\]
defines a filtration $G'^q \Gr_p^{F'} R \psi \Lambda$ of $\Gr_p^{F'} R \psi \Lambda$
by sub $(-n)$-shifted perverse sheaves.

(4)
Let $T$ be an element of the inertia group $I_K$ such that $t_l(T)$ is a generator of $\bZ_l(1)$. 
For $p \geq 0$, the map $N = T-1$ sends $F'_{p+1} R \psi \Lambda$ to $F'_p R \psi \Lambda$.
The induced map
\[
 \bar{N} \colon \Gr^{F'} _{p+1} R \psi \Lambda = R ^{p+1} \psi \Lambda [-(p+1)] \rightarrow 
 \Gr _p^{F'} R \psi \Lambda = R^p \psi \Lambda[-p]
\]
and the isomorphism in (3) make a commutative diagram
\[ \small 
\xymatrix{
R^{p+1} \psi \Lambda[-(p+1)] \ar[r]^(0.35){\bar \theta} \ar[d]^{\bar{N}} & [
0 \rightarrow i^* R^{p+2} j_* \Lambda(1) 
  \stackrel{\theta \cup}{\rightarrow}
  \cdots
  \stackrel{\theta \cup}{\rightarrow}
  i^* R^{n+1} j_* \Lambda(n-p)
] \ar[d]^{\otimes t_l(T)} \\
R^{p} \psi \Lambda[-p] \ar[r]^(0.22){\bar \theta} &  [
i^* R^{p+1} j_* \Lambda(1)
  \rightarrow i^* R^{p+2} j_* \Lambda(2) 
  \stackrel{\theta \cup}{\rightarrow}
  \cdots
  \stackrel{\theta \cup}{\rightarrow}
  i^* R^{n+1} j_* \Lambda(n-p+1)
] 
}
\]
where, in the right column, the rightmost sheaves are put on degree $n$.
\end{lem}
\begin{proof}

(1)--(3)
We note that a smooth sheaf on a pure $d$-dimensional non-singular algebraic space over a field is
a $d$-shifted perverse sheaf (as in the scheme case).
Hence $a_{p *} \Lambda[-p]$ are $(-n)$-shifted perverse sheaves.
The rest of the proof (using Lemma \ref{wss12}) is identical to the scheme case.

(4)
It suffices to show the commutativity after taking an \'etale covering.
Hence we can reduce to the scheme case.
\end{proof}

\begin{lem}
 [cf.\ \cite{Saito:weightSS}*{Corollary 2.6 (1),(2),(3),(5)}] \label{wss26}
Let $T$ and $N$ be as above.

(1)
We have $N^{n+1} = 0$.
The kernel filtration $F_p$ defined by $F_p R \psi \Lambda = \Ker (N^{p+1} \colon R \psi \Lambda \rightarrow R \psi \Lambda)$
is equal to $F'_p$.

(2)
The image filtration $G^q$ defined by 
$G^q \Gr^F_p R \psi \Lambda = \Image (\bar N^q \colon \allowbreak \Gr^F_{p+q} R \psi \Lambda \rightarrow \Gr^F_p R \psi \Lambda)$
is equal to $G'^q$.
The filtration $G^q \Gr^F_p R \psi \Lambda$ is induced by the image filtration
$G^q R \psi \Lambda = \Image (N^q \colon R \psi \Lambda \to R \psi \Lambda)$.

(3)
The isomorphism 
$\theta' \colon a_{(p+q)*} \Lambda \rightarrow i^* R^{p+q+1} j_* \Lambda(p+q+1)$
in Lemma \ref{wss12} (1) induces an isomorphism
\[
 a_{(p+q)*} \Lambda(-p)[-(p+q)]
 \rightarrow
 \Gr_{G'}^q \Gr_p^{F'} R \psi \Lambda
\]
of $(-n)$-shifted perverse sheaves.

(4)
Let $p,q \geq 0$.
The diagram
\[
\xymatrix{
 a_{(p+q)*} \Lambda(-p)[-(p+q)]
 \ar[r] \ar[d]^{1 \otimes t_l(T)} 
 &
 \Gr_{G'}^q \Gr_p^{F'} R \psi \Lambda 
 \ar[d]^{\bar N}
 \\
 a_{(p+q)*} \Lambda(-(p-1))[-(p+q)]
 \ar[r] 
 &
 \Gr_{G'}^{q+1} \Gr_{p-1}^{F'} R \psi \Lambda 
}
\]
is commutative, 
where the horizontal maps are the isomorphisms in (3).
\end{lem}
\begin{proof}
(1), (2)
Same to the scheme case: 
by using Lemmas \ref{wss12} and \ref{wss25},
we check the conditions which characterize the kernel and image filtrations.

(3) Same to the scheme case.

(4)
Clear from Lemma \ref{wss25} (4).
\end{proof}

\begin{lem}
 [cf.\ \cite{Saito:weightSS}*{Proposition 2.7}] \label{wss27}
Let $X$ be a strictly semistable algebraic space purely of relative dimension $n$ over $\Spec \cO_K$.
Let $T$ and $N$ be as above. 
Let $M_\bullet$ be the monodromy filtration on $R \psi \Lambda$ defined by the (nilpotent) operator $N$.
Then the isomorphism in Lemma \ref{wss26} (3) induces an isomorphism
\[
\bigoplus_{p-q=r} a_{(p+q)*} \Lambda(-p)[-(p+q)] \rightarrow \Gr_r^M R \psi \Lambda
\]
compatible with the action of $G_k$.
The filtration $M_\bullet$ and the canonical isomorphism are independent of the choice of $T$.

Here the increasing filtration $M_\bullet$ is defined by 
$M_r A = \sum_{p-q=r} F_p A \cap G^q A$.
\end{lem}

\begin{proof}
We combine the canonical isomorphism 
\[
\bigoplus_{p-q=r} \Gr^q_G \Gr^F_p R \psi \Lambda \rightarrow \Gr_r^M R \psi \Lambda
\]
 (\cite{Saito:weightSS}*{Corollary 2.4})
 with the isomorphisms in Lemma \ref{wss26} (3).
\end{proof}

\begin{proof}[Proof of Proposition \ref{prop:weightss}]
(1) The filtration $M_\bullet$ induces (M. Saito \cite{Saito:modulesdeHodge}*{Lemme 5.2.18}) a spectral sequence
\[
 E_1^{p,q} = H^{p+q}(X_{\overline k}, \Gr^M_{-p} R \psi \Lambda)
  \Rightarrow 
 H^{p+q}(X_{\overline k}, R \psi \Lambda).
\]
The canonical isomorphism of Lemma \ref{wss27}
shows that the left-hand side is isomorphic to 
the left-hand side of the sequence in Proposition \ref{prop:weightss}.
Therefore it remains to show 
$H^{p+q}(X_{\overline k}, R \psi \Lambda) \cong H^{p+q}(X_{\overline K}, \Lambda)$,
which is immediate if we can apply the proper base change theorem.
That theorem for algebraic spaces (more generally for stacks) is proved by Liu--Zheng \cite{Liu--Zheng}*{Theorem 0.1.1}.

(Alternatively, if $\charac K = 0$, 
we can use Artin's proof \cite{Artin:representabilite}*{Chapitre VII} of the theorem for algebraic spaces:
although his proof is stated only for those of finite type over a base which is an algebraically closed field,
it is valid when the base is an excellent Dedekind scheme (e.g. discrete valuation rings of characteristic zero).)

The compatibility with $N$ follows from Lemma \ref{wss26} (4).

(2)
First we review the following proof of Nakayama in the case $X$ is a scheme.
Equip $\Spec O_K$ with a log structure by the homomorphism
$\bN \to \cO_K \colon 1 \mapsto \pi$,
(we fix a chart $\Spec O_K \to \Spec \bZ[\bN]$), 
and $\Spec k$ with the restriction.
The special fiber $X_k$ is naturally a semistable log scheme over $\Spec k$.
Let $(X_k)_{1/l^{n}}$ be the log scheme  $X_k \times_{\Spec \bZ[\bN]} \Spec \bZ[l^{-n} \bN]$.
Let $(X_k)_{\mathrm{t} l}$ be the $2$-limit $\varprojlim_n ((X_k)_{1/l^{n}})^{\log}_\et$ of log \'etale topoi
and $\pi \colon (X_k)_{\mathrm{t} l} \to (X_k)^{\log}_\et$ be the projection.
Denote by $\varepsilon \colon (X_k)^{\log}_\et \to (X_k)_\et$ the forgetting-log morphism. 
Let $(X_k)^L$, $L$ a finite extension of $K$, 
be the log scheme obtained from $X_k$ by the base change $\Spec \cO_L \to \Spec \cO_K$.
Let $(X_k)^\tame$ be the $2$-limit $\varprojlim_L ((X_k)^L)^{\log}_\et$,
where $L$ runs over the set of tame extensions $L$ of $K$,
and $\pi^\tame \colon (X_k)^{\tame} \to (X_k)^{\log}_\et$ be the projection.
We have a natural morphism $(X_k)^\tame \to (X_k)_{\mathrm{t} l}$.
Then 
we have isomorphisms
\begin{equation} \label{eq:nearbycycles} \tag{\textasteriskcentered}
 R \varepsilon_* \pi_* \Lambda 
 \isomto R \varepsilon_* \pi^\tame_* \Lambda
 = R \varepsilon'_* \Lambda
 \isomto R \varepsilon'_* R \Psi \Lambda
 = R \psi R \varepsilon_* \Lambda 
 = R \psi \Lambda.
\end{equation} 
Here
the first morphism is induced from the adjoint property and is proved to be isomorphism in \cite{Nakayama:degeneration}*{proof of Proposition 1.9},
the second isomorphism follows from the equality 
$\varepsilon' = \varepsilon \circ \pi^\tame$ of functors 
(note the identification $X_{\overline k (\log)} \isomto X_k^\tame$ in \cite{Nakayama:nearby}*{Proposition 3.1.3}),
the third and the fourth isomorphism are respectively Theorem 3.2 and Section 3.1.6 of \cite{Nakayama:nearby},
and the fifth isomorphism (where $\varepsilon$ is the forgetting-log morphism on the generic fiber)
is trivial since the generic fiber has trivial log structure.
The isomorphisms are known to be compatible with the actions of $I_K$.
Hence we have an isomorphism between the corresponding spectral sequences.
The rightmost side gives our spectral sequence 
and the log spectral sequence obtained from the leftmost side degenerates at $E_2$ (\cite{Nakayama:degeneration}*{Theorem 2.1}).

The $E_2$-degeneration of the spectral sequence associated with the complex of log nearby cycles on the special fiber
is true in the algebraic space case
(since we assumed that $X_k$ is a scheme).
Therefore it suffices to show the isomorphism \eqref{eq:nearbycycles}
when $X$ is an algebraic space.
Take an \'etale covering $Y \to X$ by a scheme.
We have canonical isomorphisms \eqref{eq:nearbycycles} on $Y_k$, on $Y_k \times_{X_k} Y_k$, and on $Y_k \times_{X_k} Y_k \times_{X_k} Y_k$,
which are compatible with pullbacks.
Then the isomorphism on $Y_k$ descends to an isomorphism of perverse sheaves on $X_k$, 
since giving a perverse sheaf on $X_k$ is equivalent to 
giving a perverse sheaf on $Y_k$ equipped with an isomorphism between its pullbacks to $Y_k \times_{X_k} Y_k$ satisfying a certain cocycle condition.

(3)
The proof of \cite{Rapoport--Zink:monodromie}*{Satz 2.13} applies
once we have the same description of the morphism $N$ (given in (1)) and the boundary maps 
in our case.
So it suffices to show the following lemma.

(If we restrict to the case when $X_k$ is liftable to a semistable scheme over a discrete valuation ring,
which is enough for our application on K3 surfaces,
then we can reduce directly to the scheme version (\cite{Saito:weightSS}*{Proposition 2.10}) of the next lemma
since (from the above isomorphism) the spectral sequence depends only on the log special fiber.)
\end{proof}

\begin{lem}[cf.\ \cite{Saito:weightSS}*{Proposition 2.10}] \label{wss210}
The boundary map 
\[
 d_1^{p,q} \colon 
 E_1^{p,q} = \bigoplus_{i \geq \max\{0,-p\}} H^{q-2i} (X_{\overline k}^{(p+2i)}, \Lambda(-i))
 \to E_1^{p+1,q}
\]
of the weight spectral sequence is given by $\sum_{i \geq \max\{0,-p\}} \delta_{(p+2i)*} + \delta_{p+2i}^*$.
\end{lem}

Here $\delta_{\bullet, *}$ and $\delta_{\bullet}^*$ are defined as follows.
Recall that $Z_H$ is the subscheme $Z_{h_0} \cap \cdots \cap Z_{h_p}$ for each subset $H = \{ h_0, \ldots, h_p \}$ of the index set $\{ 1, \ldots, m \}$. 
For each pair of subsets $H' \subset H$ with $\lvert H' \rvert = \lvert H \rvert - 1$,
let $i_{H'H} \colon Z_H \to Z_{H'}$ be the closed immersion and,
writing $H' = \{ \ldots, \hat h_g, \ldots \} \subset H = \{ h_0, \ldots, h_p \}$ with $h_0 < \cdots < h_p$,
let $\varepsilon_{H'H} = (-1)^g$.
We define 
\[
 \delta_{p}^*
  = \sum_{H,H'} \varepsilon_{HH'} i_{HH'}^* 
  \colon H^q(X_{\overline k}^{(p)}, \Lambda) \to H^q(X_{\overline k}^{(p+1)}, \Lambda)
\]
where the sum runs pairs $H \subset H'$ with $\lvert H \rvert = p$ and $\lvert H' \rvert = p+1$, 
and 
\[
 \delta_{p*}
  = \sum_{H,H'} \varepsilon_{H'H} i_{H'H,*} 
  \colon H^q(X_{\overline k}^{(p)}, \Lambda) \to H^{q+2}(X_{\overline k}^{(p-1)}, \Lambda)
\]
where the sum runs pairs $H' \subset H$ with $\lvert H \rvert = p$ and $\lvert H' \rvert = p-1$.

\begin{proof}[Proof of Lemma \ref{wss210}]
By the definition of the spectral sequence, the map $ d_1^{p,q} \colon E_1^{p,q} \to E_1^{p+1,q}$ is
the boundary map
\[
 H^{p+q}(X_{\overline k}, \Gr^M_{-p} R \psi \Lambda) \to
 H^{p+q+1}(X_{\overline k}, \Gr^M_{-p-1} R \psi \Lambda)
\]
of the short exact sequence 
\[
 0 \to \Gr^M_{-p-1} R \psi \Lambda
 \to M_{[-p-1,-p]} R \psi \Lambda
 \to \Gr^M_{-p} R \psi \Lambda
 \to 0
\]
of $(-n)$-shifted perverse sheaves, 
where $M_{[a,b]}$ is the subquotient $M_b / M_{a-1}$
(of which $\Gr^M_a = M_{[a,a]}$ is a special case).
For an integer $q$, 
let $(K_i^j, {d'}_i^j, {d''}_i^j)$, where $K_i^j = H^{q-i+j}(X_{\overline k}, \Gr_i^F \Gr_G^j R \psi \Lambda)$, 
be the double complex where the (anti-commuting) differentials ${d'}_i^j$ and ${d''}_i^j$ are respectively the boundary maps of the short exact sequences
\[
 0 \to \Gr_{i-1}^F \Gr_G^j R \psi \Lambda
 \to F_{[i-1,i]} \Gr_G^j R \psi \Lambda
 \to \Gr_i^F \Gr_G^j R \psi \Lambda \to 0
\]
and
\[
0 \to \Gr_G^{j+1} \Gr_i^F R \psi \Lambda
 \to G^{[j,j+1]} \Gr_i^F R \psi \Lambda
 \to \Gr_G^j \Gr_i^F R \psi \Lambda \to 0.
\]
Then the complex $E_1^{\bullet,q}$ is the simple complex associated to the double complex 
 $(K_i^j, {d'}_i^j, {d''}_i^j)$.

Hence it suffices to show that the diagrams 
\begin{equation} \label{eq:commu-gysin}
\xymatrix{
 H^{q-2i}    (X_{\overline k}^{(i+j)},   \Lambda(-i))   \ar[d] \ar[r]^(0.45){\delta_{(i+j)*}} &
 H^{q-2i+2}  (X_{\overline k}^{(i+j-1)}, \Lambda(-i+1)) \ar[d] \\
 H^{q-i+j}   (X_{\overline k}, \Gr_i^F     \Gr_G^j R \psi \Lambda) \ar[r]^{d'} &
 H^{q-i+j+1} (X_{\overline k}, \Gr_{i-1}^F \Gr_G^j R \psi \Lambda)
} 
\end{equation}
and 
\begin{equation} \label{eq:commu-pullback}
\xymatrix{
 H^{q-2i}    (X_{\overline k}^{(i+j)},   \Lambda(-i))   \ar[d] \ar[r]^{\delta_{i+j}^{*}} &
 H^{q-2i}    (X_{\overline k}^{(i+j+1)}, \Lambda(-i)) \ar[d] \\
 H^{q-i+j}   (X_{\overline k}, \Gr_G^j     \Gr_i^F R \psi \Lambda) \ar[r]^{d''} &
 H^{q-i+j+1} (X_{\overline k}, \Gr_G^{j+1} \Gr_i^F R \psi \Lambda)
}
\end{equation}
commute, 
where the vertical morphisms are induced from the isomorphisms in Lemma \ref{wss26} (3).

First we consider \eqref{eq:commu-gysin}.
Since the commutativity for $(i,j)$ is equivalent to that of $(i-1,j+1)$ by Lemma \ref{wss26} (4),
we can reduce to the case $j=0$. 
In that case, it suffices to show the commutativity of the diagram
\[
\xymatrix{
0 \ar[d] & 0 \ar[d] \\
a_{(i-1)*} R i^{(i-1)!} \Lambda(1)[i+1] \ar[r] \ar[d] &
  \Gr_{i-1}^F \Gr_G^0 R \psi \Lambda \ar[d] \\
R \sHom ([\Lambda_{X_{\overline k}^{(i-1)}} \stackrel{\delta \wedge}{\to} \Lambda_{X_{\overline k}^{(i)}}], \Lambda) (1)[2] \ar[r] \ar[d] &
  F_{[i-1,i]} \Gr_G^0 R \psi \Lambda \ar[d] \\
a_{i*} R i^{(i)!} \Lambda(1)[i+2] \ar[r] \ar[d] &
  \Gr_i^F \Gr_G^0 R \psi \Lambda  \ar[d] \\
0 & 0.
}
\]
We easily reduce to the scheme case, 
which is shown in \cite{Saito:weightSS}*{Proposition 2.10}.
(It would be also possible to extend his proof directly to the algebraic space case.)

The commutativity of \eqref{eq:commu-pullback} follows
from the commutative diagram
\[
\xymatrix{
0 \ar[d] & 0 \ar[d] \\
a_{(i+j+1)*} \Lambda(-i)[-(i+j+1)] \ar[r] \ar[d] &
  \Gr_G^{j+1} \Gr_i^F R \psi \Lambda \ar[d] \\
[a_{(i+j)*} \Lambda(-i) \stackrel{\delta \wedge}{\to} a_{(i+j+1)*} \Lambda(-i) ] \ar[r] \ar[d] &
  G^{[j,j+1]} \Gr_i^F R \psi \Lambda \ar[d] \\
a_{(i+j)*} \Lambda(-i)[-(i+j)] \ar[r] \ar[d] &
  \Gr_G^j \Gr_i^F R \psi \Lambda  \ar[d] \\
0 & 0
}
\]
which follows from Lemmas \ref{wss25} (3) and \ref{wss12} (3).
\end{proof}

\section{Proof of the main theorem} \label{sec:proof}

In this section we use the following notation:
we denote 
objects (schemes, line bundles on schemes, ...) over $\cO_K$ by calligraphic letters like $\cX$,
and objects over fields by normal letters like $X$.
For example, the generic (resp.\ special) fiber of an object $\cY$ over $\cO_K$ is denoted by $Y_K$ (resp.\ $Y_k$).

We follow a method of Maulik \cite{Maulik:supersingular}*{Section 4} of studying reduction of K3 surfaces.
Let $X$ and $L$ be as in the statement of the theorem (we do not assume at this moment that $H^2$ is unramified).
Put $L^2 = 2d$ (this value is always even).

First we construct (after field extension) a projective strictly semistable scheme $\cX'$ whose generic fiber is birational to $X$.
The assumption on $L$ is used only in this step. 
(Hence, as we mentioned in the introduction, this assumption can be dropped if we admit the semistable reduction conjecture.)

\begin{lem}
Let $X$ be a K3 surface of characteristic $\neq 2$ and $L$ an ample line bundle on $X$.
Then one of the following holds:
 (a) $L$ is very ample, 
 (b) $L = \cO_X(kB)$ with $k = 1$ or $2$ and $B$ is an irreducible curve of arithmetic genus $2$, or
 (c) $X$ admits an irreducible curve $E$ of arithmetic genus $1$
  and an irreducible curve $\Gamma$ of arithmetic genus $0$ with $E \cdot \Gamma = r$ with $r = 1$ or $2$.
(The ``curves'' may be singular.)
\end{lem}
\begin{proof}
This follows from results of Saint-Donat \cite{Saint-Donat:projectiveK3} as follows. 
By \cite{Saint-Donat:projectiveK3}*{Theorem 8.1}, 
either $X$ admits an irreducible curve $E$ with $p_a(E) = 1$ and an irreducible curve $\Gamma$ with $p_a(\Gamma) = 0$ and $E \cdot \Gamma = 1$,
or $L = \cO_X(C)$ with $C$ an irreducible curve with $p_a(C) > 1$.
If the former holds then we have (c). 
Assume that the latter holds.
Note that since $X$ is a K3 surface we have $C^2 = C \cdot (C+K) = 2 p_a(C) - 2$.
By \cite{Saint-Donat:projectiveK3}*{Theorem 3.1}, $L$ is base-point free
and hence we have a morphism $\phi_L \colon X \to \bP^N$.
Since $L^2 > 0$ the image of $\phi_L$ is $2$-dimensional.
By \cite{Saint-Donat:projectiveK3}*{Section 4.1}, the degree of $\phi_L$ is either $1$ or $2$, 
and if it is of degree $2$ then (b) or (c) holds
by \cite{Saint-Donat:projectiveK3}*{Section 5.1} (if $C^2 = 2$) and \cite{Saint-Donat:projectiveK3}*{Theorem 5.2} (if $C^2 > 2$).
If $\phi_L$ is of degree $1$ then it is an embedding
since it contracts no curve (since $L$ is ample), 
and this means that $L$ is very ample.
\end{proof}

We first consider case (a).
We embed $X$ into a projective space $\bP^N$ by $\lvert L \rvert$, 
and then take the composite with a projection $\bP^N \rationalto \bP^1$ (which is a rational map) in a general position. 
We can resolve the points of indeterminacy of the rational map $X \rationalto \bP^1$
and obtain a morphism $X' \to \bP^1$.
By \cite{Maulik:supersingular}*{Lemma 4.2}, 
all the fibers of this morphism are nodal, 
and general fibers are irreducible of genus $g = d+1$
(here we need $p > d+4$).
By \cite{Maulik:supersingular}*{Remark 4.3}
we can assume that general fibers are smooth.
Then by \cite{Saito:logsmooth}*{Corollary 1.9}, 
after replacing $K$ by a finite extension, 
we obtain a projective strictly semistable model $\cX'$ of $X'$ over $\cO_K$.
(Here we need $p > 2g+2 = 2d+4$.)

In case (b), 
the construction is similar\footnote{
In detail: Let $C = (f=0) \subset \bP^2$ be the ramification divisor of $X \to \bP^2$.
Since $p > B^2 + 4 = 6 = \deg f$ (since $B$ is the pullback of $\cO(1)$) and $C$ is smooth, 
$C$ has only finitely many inflection points.
Take a point on $\bP^2$ which is not on the union of $C$ and the tangent lines at the inflection points,
and take the projection from that point.
Then all the fibers of the resulting fibration $X' \to \bP^1$ are nodal
and general fibers are smooth irreducible.
Now use Saito's result similarly.}
 to the very ample case, 
in which we use the morphism $\lvert B \rvert \colon X \to \bP^2$ in place of $X \hookrightarrow \bP^N$.
(We need the inequality $p > 2 + 4$, 
which is satisfied since $p > 2d + 4 \geq 2 + 4$.)

In case (c),
it follows \cite{Pjateckii-Shapiro--Shafarevich}*{Theorem 1 of Section 3} 
(this requires $p \geq 5$)
that the linear system $\lvert E \rvert$ induces an elliptic fibration $X \to \bP^1$
with smooth general fibers.
Then $X \setminus \Gamma \to \bP^1$ is a hyperbolic fibration 
(whose general fibers are $r$-punctured elliptic curves).
Hence we can apply \cite{Saito:logsmooth}*{Corollary 1.9} similarly.
(We need the inequality $p > 2g + 2 = 4$ and $p > r$, 
which is satisfied since $p > 2d + 4 \geq 6$.)

Applying the minimal model program (\cite{Kawamata:mixed3fold}; this requires $p \geq 5$) to this strictly semistable scheme $\cX'$, 
we obtain a ``minimal model'', 
a projective flat scheme $\cX''$ over $\cO_K$
satisfying the following properties:
the generic fiber is smooth and birational to $X'$ (hence to $X$),
the irreducible components of the special fiber are geometrically normal, 
the relative canonical divisor $K_{\cX''/\cO_K}$ is nef and $\bQ$-Cartier, 
and $(\cX'', X''_k)$ has (at worst) log terminal singularities.
(Here $K_{\cX'' / \cO_K}$ is by definition
the Weil divisor, defined up to linear equivalence, 
which agrees with $\bigwedge^2 \Omega^1_{(\cX'')^\sm / \cO_K}$
on the smooth part $(\cX'')^{\sm}$.
This is well-defined since $\cX'' \setminus (\cX'')^\sm$ is of codimension at least $2$.)

Then it follows (since $X$ is K3, see \cite{Maulik:supersingular}*{Section 4.3}) that
the generic fiber $X''_K$ is isomorphic to $X$ 
and $K_{\cX'' / \cO_K} = 0$.

We apply Kawamata's classification \cite{Kawamata:mixed3fold}*{Theorem 4.4} of log terminal singularities (of index $1$): 
every non-smooth point of $\cX''$ is one of the following types.
(1) Semistable singularity (i.e., \'etale locally of the form $\cO_K[x,y,z] / (xy-\pi)$ or $\cO_K[x,y,z]/(xyz-\pi)$).
(2) An isolated non-smooth point which is a rational double point in the special fiber $X''_k$.

Moreover, the irreducible components of the special fiber are normal (by the construction of the minimal model program)
and hence regular in codimension one.
Hence it follows that $\cX''$ is \emph{strictly} semistable away from points of type (2).

We note that, if $\cX''$ is such a model over $\cO_K$, 
then 
for any extension $K'$ of $K$
we can construct a model over $\cO_{K'}$ satisfying the same properties.
This follows from the result of Saito \cite{Saito:logsmooth}*{Theorem 2.9.2}
that there exists a log blow-up $\cY \to \cX'' \otimes^{\log}_{\cO_K} \cO_{K'}$
such that $\cY$ is strictly semistable over $\cO_{K'}$ away from points of type (2).
Then since both log base change and log blow-up preserve the sheaf $\bigwedge^2 \Omega^1(\log)$ of top log differentials, 
and since the canonical divisor on $\cX''$ and $\cY$ corresponds respectively to
the line bundles 
$\bigwedge^2 \Omega^1_{(\cX'')^\sm/O_K}(\log)$ and
$\bigwedge^2 \Omega^1_{\cY^\sm/O_{K'}}(\log)$,
it follows that $K_{\cY/\cO_{K'}} = 0$.
The other properties are immediate.

By \cite{Artin:brieskorn}*{Theorem 2}, 
singularities of type (2) can be resolved potentially \emph{in the category of algebraic spaces}.
That is, after we replace $K$ by a finite extension (and replace $\cX''$ as above), there exists an algebraic space $\cX'''$ 
and a morphism $\phi \colon \cX''' \to \cX''$ satisfying the following conditions:
\begin{itemize}
\item $\phi$ is an isomorphism outside the singularities of type (2), 
and 
\item For each irreducible component $Z$ of $X''_k$ 
(note that $Z$ is smooth outside points of type (2)), 
$\phi\rvert_{\phi^{-1}(Z)} \colon \phi^{-1}(Z) \to Z$ is the minimal desingularization.
\end{itemize}

$\cX'''$ is an algebraic space over $\cO_K$ having (at worst) strictly semistable singularities.
Then the special fiber $X'''_k$ is a scheme, 
since it is covered by two open subschemes:
$(X'''_k)^\sm$ (which is a scheme since smooth $2$-dimensional)
and the complement of rational double points in $X''_k$.
Therefore $X'''_k$ is an SNC surface. 
It is projective since it is a blow-up of the projective scheme $X''_k$.
We want to show that this is an SNC log K3 surface (in the sense of Nakkajima \cite{Nakkajima:logk3}), 
that is, $\bigwedge^2 \Omega^1_{X'''_k/k}(\log)$ is trivial and $H^1(X'''_k, \cO_{X'''_k}) = 0$.

Since $\phi_k \colon X'''_k \to X''_k$ is the blow-up of the rational double points, 
we have 
$\phi_k^* \colon H^1(X'''_k, \cO_{X'''_k}) \cong H^1(X''_k, \cO_{X''_k})$.
Since $X''_k$ is Gorenstein, 
the dualizing complex is represented by an invertible sheaf $L$.
Let $\cU$ be the complement of the points of type (2) in $\cX''$.
Then, since $\cU$ is log smooth over $\cO_K$, we have $\bigwedge^2 \Omega^1_{U_k/k}(\log) = \bigwedge^2 \Omega^1_{\cU/\cO_K}(\log) \rvert_{U_k}$, 
which is trivial.
Since $L \rvert_{U_k}$ is isomorphic to this log canonical sheaf (\cite{Tsuji:Poincare}*{Theorem 2.21}),
and since $X''_k \setminus U_k$ is of codimension at least $2$, 
$L$ itself is trivial.
Then by duality we have $\dim H^2(X''_k, \cO_{X''_k}) = \dim H^0(X''_k, \cO_{X''_k})^\vee = 1$.
Since the Euler--Poincar\'e characteristic of $X''_k$ (which is equal to that of the generic fiber) is $2$,
we obtain $H^1(X''_k, \cO_{X''_k}) = 0$, 
hence $H^1(X'''_k, \cO_{X'''_k}) = 0$.
Since $\bigwedge^2 \Omega^1_{U_k/k}(\log)$ is trivial, 
and since the resolution of rational double points (which are canonical singularities) does not change the canonical divisor,
$\bigwedge^2 \Omega^1_{X'''_k/k}(\log)$ is also trivial.
Thus $X'''_k$ is an SNC log K3 surface.

Nakkajima \cite{Nakkajima:logk3}*{Proposition 3.4} gave a classification of SNC log K3 surfaces 
(which is parallel to Kulikov's classification in the complex case \cite{Kulikov:degeneration}*{Theorem II})
in arbitrary characteristic.
Using that, we obtain the following list of possibilities for the shape of $X'''_k$ (after replacing $k$ by an algebraic closure):

Type I: A smooth K3 surface.

Type II: A union of surfaces $Z_1, \ldots, Z_m$
with $Z_1$ and $Z_m$ rational and others elliptic ruled. 
Double curves
 $Z_h \cap Z_{h'}$ ($h \neq h'$) are rulings (elliptic curves) if $\lvert h-h' \rvert = 1$ and empty otherwise.
(There are no triple points.)

Type III: A union of rational surfaces,
whose dual graph of the configuration is a triangulation of $S^2$ (the sphere).

Now we use the unramified/crystalline hypothesis.
Applying the comparison theorems (Propositions \ref{prop:p-ss}, \ref{prop:weightss}) on $\cX'''$,
we observe that $E_2^{1,1}$ and $E_2^{2,0}$ are zero if $H^2$ of the generic fiber is unramified/crystalline.
Therefore it suffices to show that
if the special fiber is of Type II (resp.\ III) then $E_2^{1,1}$ (resp.\ $E_2^{2,0}$) is nonzero.

We can deduce this from the above description
and the description of the map $d_1$ (given in Lemma \ref{wss210} in the $l$-adic case
and in \cite{Mokrane:spectrale}*{Corollaire 4.14} in the $p$-adic case).
We write the proof in the $l$-adic notation (the proof in the $p$-adic case is identical).
Type II: 
observe that $E_1^{0,1} = H^1(X_k''^{(0)}, \Lambda)$ and $E_1^{1,1} = H^1(X_k''^{(1)}, \Lambda)$ 
are the direct sums of $\Lambda^{\oplus 2}$ respectively for each non-rational component and each double curve.
Hence $E_2^{1,1} = \Coker (\Res \colon E_1^{0,1} {\to} E_1^{1,1}) \neq 0$.
Type III: 
$E_1^{1,0} = H^0(X_k''^{(1)}, \Lambda)$ and $E_1^{2,0} = H^0(X_k''^{(2)}, \Lambda)$ 
are the direct sums of $\Lambda$ respectively for each double curve and each triple point.
Then $E_2^{2,0} = \Coker (\Res \colon E_1^{1,0} {\to} E_1^{2,0})$
is isomorphic to $H^2(S^2, \Lambda)$ (singular cohomology), which is nonzero.

Thus Theorem \ref{thm:maintheorem} is proved.
For Remark \ref{rem:maintheorem} (3), 
$\cX''$ is a (scheme) model of $X$ with only rational double point singularities in the special fiber.

\medskip

For the later application 
we need the following refinement.
Recall that 
at each step of the minimal model program
we have an opportunity of choosing which extremal ray to contract.

\begin{prop}[\cite{Maulik:supersingular}*{Theorem 4.1}] \label{prop:quasipolarization}
Assume that $L$ is very ample (and that we applied the construction of case (a)).
By a suitable choice of the extremal ray at each step of the minimal model program,
we can assume that the resulting model $\cX'''$
admits a quasi-polarization 
which extends 
the polarization on $X$ defined by $L$.
\end{prop}
By definition a quasi-polarization of $\cX'''$ is an element of $\Pic(\cX''')$
whose restriction to each geometric fiber is
the class of a nef big line bundle in the Picard group.

\begin{proof}
It suffices to extend $L$ to quasi-polarization on $\cX''$.
This is achieved by applying ``minimal model program with scaling'', 
as explained in \cite{Maulik:supersingular}*{Section 4.3}.
Although his theorem is stated for the case $X$ is supersingular,
his argument can be applied to our more general case (provided $X$ has potential good reduction).
\end{proof}

This proposition fails in the case $L$ is not very ample.
The problem is that, in case (c), 
the fibration we used (which is induced by $\lvert E \rvert$) is different from the one induced by $\lvert L^{\otimes m} \rvert$ ($m$ large enough),
and Maulik's argument applied to this fibration extends only $E$, not $L$, to a nef divisor on $\cX''$.

\section{Moduli spaces and period maps} \label{sec:moduli}

The (potential) good reduction criterion is deeply related to the surjectivity of the period maps of K3 surfaces
(I thank Tetsushi Ito and Keerthi Madapusi Pera for explaining this concept to me).
In this section we prove this surjectivity 
and, as a by-product, give a bound for the extension degree in Theorem \ref{thm:maintheorem}.

First we shall introduce the 
moduli stacks of K3 surfaces, orthogonal Shimura varieties, period maps, and those with level structures, and their integral models.
For precise definitions and proofs see 
\cite{Rizov:Kuga--Satake}, \cite{Maulik:supersingular}, \cite{MadapusiPera:TateK3} and \cite{MadapusiPera:integralmodels}\footnote{
One should be careful since the notation differs in these papers.
We mainly follow that of \cite{MadapusiPera:TateK3}, 
but in order to avoid collision of notation we use $\bK$ instead of his $K$ and $\Lambda_{d}$ instead of his $L_{d}$.
}.

Let $d$ be a positive integer (which we fix throughout this section). 
A K3 surface over a scheme $S$ is a smooth proper \emph{algebraic space} over $S$ 
whose fibers are K3 surfaces (over fields, in the usual sense).
A primitive quasi-polarization (resp.~a primitive polarization) of a K3 surface $X$ over $S$
is a section $\xi \in \Gamma(S, \underline{\Pic}(X/S))$
whose restriction $\xi_{\overline k}$ on each geometric fiber is the class of a nef big line bundle (resp.~an ample line bundle) on $X_{\overline k}$
which is primitive (i.e., $\xi_{\overline k}$ is not a nontrivial multiple of an element of $\Pic(X_{\overline k})$).
The degree of $\xi$ is the self-intersection $\xi^2$, which is a locally constant integer on $S$.
We denote by $M_{2d}$ the Deligne--Mumford stack over $\bZ[1/2]$ parametrizing 
K3 surfaces equipped with primitive quasi-polarizations of degree $2d$, 
and by $M^\circ_{2d}$ its substack where the quasi-polarization is a polarization.
By definition, for a scheme $S$, 
$M_{2d}(S)$ (resp.~$M^\circ_{2d}(S)$) is the groupoid\footnote{A groupoid is a category such that all morphisms are isomorphisms. 
A set can be naturally regarded as a groupoid, 
but the groupoids $M_{2d}(S)$ and $M^\circ_{2d}(S)$ are not of that kind.
} whose objects are the K3 surfaces over $S$ equipped with a primitive quasi-polarization (resp.~a primitive polarization) of degree $2d$.
Then $M_{2d}$ and $M^\circ_{2d}$ are of finite type over $\bZ[1/2]$.

Let $U = \langle e, f \rangle$ be the quadratic lattice ($=$ $\bZ$-module equipped with a quadratic form)
of rank $2$ with $e \cdot e = f \cdot f = 0$ and $e \cdot f = 1$,
and $E_8$ the $E_8$ lattice.
Let $\Lambda = U^{\oplus 3} \oplus E_8 ^{\oplus 2}$
and $\Lambda_{d} = \langle e+df \rangle \oplus U^{\oplus 2} \oplus E_8 ^{\oplus 2} = \langle e-df \rangle ^\perp \subset \Lambda$.
Then
for any complex K3 surface $X$, 
there exists a (non-canonical) isometry from $H^2(X, \bZ)$ to $\Lambda$.
Here the quadratic form on $H^2(X, \bZ)$ is defined to be the canonical pairing
$H^2(X, \bZ) \times H^2(X, \bZ) \stackrel\cup\to H^4(X, \bZ) \stackrel\sim\to \bZ$
multiplied by $-1$.
Moreover,
for any primitive quasi-polarization $L$ of $X$ of degree $2d$,
we can choose such an isometry to take $c_1(L)$ to $e-df \in \Lambda$
(so that we obtain an isometry $PH^2((X,L), \bZ) \stackrel\sim\to \Lambda_d$,
where $PH^2((X, L)) = \langle c_1(L) \rangle ^\perp \subset H^2(X)$). 
Let $\Sh(\Lambda_{d})$ be the (canonical model defined over $\bQ$ of the) orthogonal Shimura variety attached to the group $\SO(\Lambda_{d} \otimes \bQ)$,
so that $\Sh(\Lambda_{d})_\bC$ parametrizes Hodge structures of a certain type on $\Lambda_d$.
Then the period map\footnote{
 To be precise, the period map is defined only on a suitable double covering $\tilde M_{2d}$ of $M_{2d}$. 
 However, if $\bK$ is neat, then  $\tilde M_{2d, \bK} \to M_{2d, \bK}$ admits a (non-canonical) section, 
 and the level structured period map $\iota_\bK$ is indeed defined on $M_{2d,\bK}$ via that section.
 Since  we actually use only $\iota_\bK$ for such $\bK$'s,
 we omit this tilde for simplicity.
}
 $\iota_\bC \colon M_{2d,\bC} \to \Sh(\Lambda_{d})_\bC$, 
attaching the Hodge structure $PH^2((X, L), \bZ)$ to each $(X, L)$, 
descends to $\iota_{\bQ} \colon M_{2d,\bQ} \to \Sh(\Lambda_{d})$.

Let $\bK_{\Lambda_{d}} \subset \SO(\Lambda_{d})(\bA_f)$ be the subgroup 
of the elements which preserve $\Lambda_d \otimes \hat \bZ$ and 
act trivially on the discriminant group $\disc \Lambda_d = \Lambda_d^{\vee} / \Lambda_d$, 
where $\Lambda_d^{\vee} = \Hom(\Lambda_d, \bZ)$ is the dual lattice.
An admissible subgroup of $\SO(\Lambda_{d})(\bA_f)$ is a compact open subgroup of $\bK_{\Lambda_{d}}$.
Let $\bK$ be an admissible subgroup.
We say that $\bK$ is neat if, for every $g \in \SO(\Lambda_{d})(\bA_f)$, the discrete group $\SO(\Lambda_{d})(\bQ) \cap g \bK g^{-1}$ is torsion-free.
We say that $\bK$ is prime to $p$ if $\bK$ is of the form $\bK^p \bK_p$ with $\bK_p = \bK_{\Lambda_{d},p}$.
(Here, as in the standard notation,
$-_p$ and $-^p$ stands for the $p$-part and the prime-to-$p$ part respectively.)

Let $p$ be an odd prime, 
and $\bK$ an admissible subgroup prime to $p$.
For a morphism $S \to M_{2d,\bZ_{(p)}}$ corresponding to $(f \colon X \to S, \xi)$ with $S$ a scheme, 
let $I^p(S)$ be the set of isometries $\underline{\Lambda \otimes \hat \bZ^p} \to R^2 f_* \hat \bZ^p(1)$ taking $e-df$ to $c_1(\xi)$
(where $\hat \bZ^p$ is the prime-to-$p$ part of $\hat \bZ$).
This defines a sheaf $I^p$ on $M_{2d,\bZ_{(p)}}$
on which $\bK_{\Lambda_{d}}^p$ acts naturally.
We define a $\bK^p$-level structure of a K3 surface $(f \colon X \to S, \xi)$ over $\bZ_{(p)}$
to be a section of the sheaf $I^p/\bK^p$ over $S$. 
Then there is a moduli stack $M_{2d, \bK, \bZ_{(p)}}$ parametrizing objects equipped with $\bK^p$-level structures,
and there is a finite \'etale map $M_{2d, \bK, \bZ_{(p)}} \to M_{2d, \bZ_{(p)}}$ of degree $[\bK_{\Lambda_{d}} : \bK]$.
We also have the Shimura variety (with level structure) $\Sh_\bK(\Lambda_{d})$ over $\bQ$,
which is a finite \'etale cover of $\Sh(\Lambda_{d})$.
The period map $\iota_{\bQ} \colon M_{2d,\bQ} \to \Sh(\Lambda_{d})$
lifts to the period map (with level structure) $\iota_{\bK,\bQ} \colon M_{2d, \bK, \bQ} \to \Sh_\bK(\Lambda_{d})$.
If $\bK$ is neat, 
then $\Sh_\bK(\Lambda_{d})$ is a scheme (from general theory).
If $\bK$ is small enough (so that there are no nontrivial automorphisms of quasi-polarized K3 surfaces with $\bK^p$-level structures)
then $M_{2d, \bK, \bZ_{(p)}}$ is an algebraic space.

Madapusi Pera (\cite{MadapusiPera:integralmodels}) recently constructed integral canonical models 
$\cS(\Lambda_{d})$ of $\Sh(\Lambda_{d})$ over $\bZ[1/2]$ and, 
for each $p>2$ and for $\bK$ prime to $p$, 
$\cS_\bK(\Lambda_{d})_{(p)}$ of $\Sh_\bK(\Lambda_{d})$ over $\bZ_{(p)}$.
If $\bK^p$ is small enough then $\cS_\bK(\Lambda_{d})_{(p)}$ is a scheme.
He extended the period maps to $\iota_{\bZ[1/2]} \colon M_{2d, \bZ[1/2]} \to \cS(\Lambda_{d})$
and $\iota_{\bK, \bZ_{(p)}} \colon M_{2d, \bK, \bZ_{(p)}} \to \cS_\bK(\Lambda_{d})_{(p)}$,
and showed that these maps are \'etale
(\cite{MadapusiPera:TateK3}*{Proposition 4.7 and its proof}).

It is known that $\iota_\bC$ (and hence $\iota_\bQ$) is surjective 
(Kulikov \cite{Kulikov:surjectivity}: this follows from his result on degenerations from arguments similar to below).
We show (under an assumption) that this is true also in characteristic $p$.

\begin{thm} \label{thm:surjectivity}
Assume $p > 18d+4$.
Take an admissible compact open prime-to-$p$ subgroup $\bK \subset \SO(\Lambda_{d})(\bA_f)$ small enough so that 
$M_{2d, \bK, \bZ_{(p)}}$ is an algebraic space and $\cS_\bK(\Lambda_{d})_{(p)}$ is a scheme.
Then $\iota_{\bK, \bZ_{(p)}} \colon M_{2d, \bK, \bZ_{(p)}} \to \cS_\bK(\Lambda_{d})_{(p)}$ is surjective.
\end{thm}
This follows from the following property of $\iota_{\bK,\bZ_{(p)}}$.
\begin{prop} \label{prop:extension}
Let $p$, $d$, and $\bK$ be as in the above theorem.
Then $\iota_{\bK,\bZ_{(p)}}$ satisfies the following extension property:
given any commutative diagram
\[
\xymatrix{
M^\circ_{2d, \bK, \bZ_{(p)}} \ar[r] & M_{2d, \bK, \bZ_{(p)}} \ar[r]^{\iota_{\bK,\bZ_{(p)}}} & \cS_\bK(\Lambda_{d})_{(p)} \\
\Spec K \ar[u] \ar[rr] & & \Spec \cO_K \ar[u]
}
\]
with $K$ a complete discrete valuation field with perfect residue field, 
there exists a finite unramified extension $K'$ of $K$ and a (not necessarily unique) morphism $\Spec \cO_{K'} \to M_{2d, \bK, \bZ_{(p)}}$ 
making the following diagram commutative:
\[
\xymatrix{
M^\circ_{2d, \bK, \bZ_{(p)}} \ar[r] & M_{2d, \bK, \bZ_{(p)}} \ar[r]^{\iota_{\bK,\bZ_{(p)}}} & \cS_\bK(\Lambda_{d})_{(p)} \\
\Spec K \ar[u] \ar[rr] & & \Spec \cO_K \ar[u] \\
\exists \Spec K' \ar[u] \ar[rr] & & \Spec \cO_{K'} \ar[u] \ar@{-->}[uul]^(0.4){\exists} .
}
\]
\end{prop}

For simplicity, we write the maps by $M^\circ_{\bK} \to M_{\bK} \stackrel{\iota_{\bK}}\to \cS_{\bK}$ or
by $M^\circ \to M \stackrel\iota\to \cS$. 

\begin{proof}[Proof of Proposition \ref{prop:extension}]
The morphism $\Spec K \to M$ corresponds to a primitively quasi-polarized K3 surface $(X, \xi)$ over $K$ with a $\bK^p$-level structure.
There is a ``Kuga--Satake'' abelian variety $A$ of $(X, \xi)$,
whose construction we recall in the next paragraph,
defined over a finite extension $K'$ of $K$
and equipped with an action of $C(\Lambda_d)$,
satisfying the following properties (as in \cite{MadapusiPera:TateK3}*{Theorem 4.17}).
For each $l \neq p$,
there exists
an isomorphism
\[ 
 H^1_\et(A_{K'^\sep}, \bQ_l) \cong C(PH^2_\et(X_{K'^\sep}, \bQ_l)(1))
\]
of $\bQ_l$-vector spaces,
where 
$C$ denotes the Clifford algebra of a quadratic space and 
$PH^2(-)(1)$ denotes the orthogonal complement of $\langle c_1(\xi) \rangle$ in $H^2(-)(1)$,
in such a way that 
the subalgebra $C(\Lambda_d) \otimes \bQ_l \subset \End H^1_\et(A_{K'^\sep}, \bQ_l) $
is Galois-equivariantly identified under the above isomorphism with the subalgebra
$C(PH^2_\et(X_{K'^\sep}, \bQ_l)(1)) \allowbreak \subset \allowbreak \End C(PH^2_\et(X_{K'^\sep}, \bQ_l)(1))$
(acting on itself by right translation).
There are also analogous isomorphisms for $p$-adic representations (if $\charac K = 0$).

We shall recall the construction of $A$
(for details see \cite{MadapusiPera:integralmodels}*{Section 3}).
There are homomorphisms $\GSpin(\Lambda_{d}) \to \SO(\Lambda_{d})$
and $\GSpin(\Lambda_{d}) \to \GSp(C(\Lambda_{d}))$,
where $C(\Lambda_{d})$ is equipped with a certain symplectic form.
These homomorphisms induce finite morphisms between the corresponding Shimura varieties and
finite morphisms between their integral models.
Moreover, denoting by $\tilde \Sh = \tilde \Sh_{\tilde \bK}(\Lambda_{d})$ and $\tilde \cS = \tilde \cS_{\tilde \bK}(\Lambda_{d})$ 
the GSpin Shimura variety and its integral model
(where $\tilde \bK$ is the inverse image of $\bK$),
we have finite \'etale morphisms $\tilde \Sh_{\tilde \bK}(\Lambda_{d}) \to \Sh_{\bK}(\Lambda_{d})$
and $\tilde \cS_{\tilde \bK}(\Lambda_{d})_{(p)} \to \cS_{\bK}(\Lambda_{d})_{(p)}$.
Replacing $K$ by some finite extension $K'$ such that
$\Spec K' \to \cS$ lifts to a morphism $\Spec K' \to \tilde \cS$,
we define the Kuga--Satake abelian variety $A$ 
to be the restriction of the pullback of the universal abelian variety on the GSp Shimura variety to that $K'$-valued point of $\tilde \cS$.
This $A$ is known to be independent up to isomorphism of the choice of the lift.

Since $\tilde \cS \to \cS$ is proper, 
$\Spec K' \to \tilde \cS$ extends to a morphism $\Spec \cO_{K'} \to \tilde \cS$.
This means that $A$ extends to an abelian scheme over $\Spec \cO_{K'}$ (again the pullback of the universal abelian variety).
Hence $H^1(A)$ is unramified (as a representation of $G_{K'}$).
Then $H^2(X)$ is also unramified: 
$PH^2(X)$ is unramified since $PH^2(X) \subset C(PH^2(X)) \cong C(\Lambda_d) \subset \End H^1(A)$
(via the isomorphism above), 
and $c_1(L)$ is Galois-invariant (since $L$ is defined over $K$).

We have $p > 18d + 4 = (L^{\otimes 3})^2 + 4$,
and 
by \cite{Saint-Donat:projectiveK3}*{Theorem 8.3} $L^{\otimes 3}$ is very ample.
Therefore, applying Theorem \ref{thm:maintheorem} and Proposition \ref{prop:quasipolarization} to the pair $(X, L^{\otimes 3})$,
after replacing $K'$ by a further finite extension $K''$ 
we obtain a proper smooth model $\cX$ of $X$ equipped with a quasi-polarization $\cL_3$ extending $L^{\otimes 3}$.
Since the closure $\cL$ of $L$ in $\cX$ satisfies $\cL^{\otimes 3} = \cL_3$,
it follows that $\cL$ is itself a quasi-polarization and extends $L$.
Also the level structure extends naturally,
and hence we obtain a desired morphism $\Spec \cO_{K''} \to M$. 
The commutativity follows easily.

(In this proof we used Theorem \ref{thm:maintheorem} for one $l \neq p$.
If $\charac K = 0$ we also could have used the $p$-adic criterion.)
\end{proof}

\begin{proof}[Proof of Theorem \ref{thm:surjectivity}]
Since the image of $\iota$ is a dense open subset of $\cS$ (since $\iota \rvert_{M^\circ}$ is an open immersion and $\iota$ is \'etale),
it suffices to show that for every closed point $s \in \cS$ 
there exists a morphism $\Spec \cO_K \to \cS$ from a discrete valuation ring
taking the closed point to $s$ and the generic point into $\Image \iota \rvert_{M^\circ}$.
Considering (Zariski-)locally, we may assume $\cS = \Spec A$, $A$ a Noetherian integral domain, 
$\Image \iota \rvert_{M^\circ} \supset \Spec A[1/f]$, $f \neq 0$,  
and $s$ corresponds to a prime $\fp \subset A$, 
and what we want to show is that $A$ admits a prime $\fq \subset \fp$ with $f \not\in \fq$ and $\height \fq = \height \fp - 1$. 
If $\height \fp = 1$ then we can take $\fq = (0)$.
If $\height \fp \geq 2$, 
then there exists infinitely many prime ideals of height $1$
and only finitely many can contain $f$, 
hence we can take $\fq \subset \fp$ with $f \not\in \fq$ and $\height \fq = 1$. Now consider $A/\fq$ and use induction on the dimension.
\end{proof}

\begin{rem}
%
%
As Keerthi Madapusi Pera explained to me,
using the Kuga--Satake abelian variety $A$ and the isomorphisms of cohomology groups, 
we can show that the $l$-adic potential good reduction criteria for primes $l \neq p$ are all equivalent to each other,
and if $\charac K = 0$ also to the $p$-adic criterion.
(These equivalences does not need the assumption on the degree of $L$.)

It suffices to show that 
the conditions that $H^2_\et(X_{K'^\sep}, \bQ_l)$ are unramified (as representations of $G_{K'}$)
are equivalent for all $l \neq p$,
and if $\charac K = 0$ also to the condition that $H^2_\et(X_{K'^\sep}, \bQ_p)$ is crystalline.
Using the construction of $A$, we can reduce to the corresponding equivalence for $H^1$'s of $A$,
which follows from the ($l$-adic and $p$-adic) good reduction criteria for abelian varieties.
\end{rem}

\begin{cor} \label{cor:mainthm-degree2}
For each $d$ there exists a (non-explicit) constant $C$ satisfying the following property:
For any $(X,L)$ as in Theorem \ref{thm:maintheorem} with $L^2 = 2d$,
if an additional condition $p > 18d + 4 = 9L^2 + 4$ is satisfied,
then the extension $K'/K$ in the theorem can be taken to be of degree dividing $C$.
\end{cor}

\begin{proof}
Fix $d$.
Let $\bK \subset \SO(\Lambda_{d})(\hat \bZ)$ be the group 
\[
 \bK = \bK(3) = \{ g \in \SO(\Lambda_{d})(\hat \bZ) : g \equiv 1 \pmod 3 \} \cap \bK_{\Lambda_d}.
\]
Then $\bK$ is admissible, prime to $p$ if $p \neq 3$, 
and neat (since $1 + 3 M(N, \bZ_3)$ is torsion-free), 
and we can show as in \cite{Maulik:supersingular}*{Proposition 2.8} that $M_{2d, \bK, \bZ[1/6d]}$ is an algebraic space.
We fix an \'etale covering $M'_{\bZ[1/6d]} \to M_{2d, \bK, \bZ[1/6d]}$ by a scheme.
Let $n_1, \ldots, n_k$ be the degrees of the connected components of $M'_{\bZ[1/6d]}$ over their images, 
and let $C_1 = \lcm\{1, 2, \ldots, \max\{ n_1, \ldots, n_k \} \}$.
Let $C_2 = \lvert \GL(21, \bF_3) \rvert$.
We prove that, given $(X, L)$ over $K$ as in Theorem \ref{thm:maintheorem}
with $L^2 = 2d$ satisfying an additional assumption $p > 18d + 4$ (hence in particular $p$ does not divide $6d$), 
$X$ has good reduction over an extension of $K$ of degree dividing $C = C_1 C_2$
(note that this $C$ depends only on $d$).

Take an extension $K_1/K$ such that $X_{K_1}$ admits a $\bK^p$-level structure
(this extension can be taken to be of degree dividing $[\bK_{\Lambda_d} : \bK]$, hence dividing $C_2$),
so that we obtain a morphism $x \colon \Spec K_1 \to M_\bK = M_{2d, \bK, \bZ_{(p)}}$.

Lifting $\iota(x) \colon \Spec K_1 \to \cS_{\bK}(\Lambda_{d})_{(p)}$ to a point of $\tilde \cS_{\tilde \bK}(\Lambda_{d})_{(p)}$
(after replacing $K_1$ by a finite extension $K_2$),
we obtain the Kuga--Satake abelian variety $A$ of $X$.
Since $H^1(A_{\overline K}, \bQ_l) = C(PH^2(X_{\overline K}, \bQ_l))$ is unramified/crystalline
 (since $H^2(X_{\overline K}, \bQ_l)$ is so),
$A$ has good reduction,
and the model of $A$ over $\cO_{K_2}$ is the pullback of the universal abelian scheme.
Since $\tilde \cS_{\tilde \bK}(\Lambda_d)_{(p)}$ is finite over the integral model of the GSp Shimura variety,
this gives a $\cO_{K_2}$-valued point of $\tilde \cS_{\tilde \bK}(\Lambda_d)_{(p)}$
and hence a $\cO_{K_2}$-valued point of $\cS_\bK(\Lambda_d)_{(p)}$.
Applying Proposition \ref{prop:extension},
for some finite extension $K_3/K_2$ 
we obtain a $\cO_{K_3}$-valued point of $M_\bK$
compatible with the $K_1$-valued point $x$. 


Put $M' = M'_{\bZ[1/6d]} \otimes \bZ_{(p)}$.
Put $M'' = M' \times_{M_{\bK}} \Spec \cO_{K_3}$ (this is a scheme).
Note that $M'' \to \Spec \cO_{K_3}$ is \'etale and surjective.
Take a point $z \in M''$ mapping to the closed point of $\Spec \cO_{K_3}$. 
Then the morphism $\Spec \cO_{M'',z} \to \Spec \cO_{K_3}$ is \'etale and surjective.
The tensor product $\cO_{M'', z} \otimes_{\cO_{K_3}} K_3$ is a finite product of finite separable extensions of $K_3$.
Take one of them, say $K_4$.
Then the morphism $\Spec K_4 \to \Spec \cO_{M'', z}$ 
factors through $\Spec \cO_{K_4} \to \Spec \cO_{M'', z}$ (by the assumption on $z$).
Also, the morphism $\Spec K_4 \to M'$ factors through $M' \times_{{M_\bK},x} \Spec K_1$
and hence through a component $\Spec K_5$ of $M' \times_{{M_\bK},x} \Spec K_1$.
Since the morphism $\Spec K_4 \to M'$ \emph{of schemes} factors through both $\Spec \cO_{K_4}$ and $\Spec K_5$,
it follows (from an elementary set-theoretic argument on the corresponding rings)
that the morphism factors through $\Spec \cO_{K_5}$.
This gives a $\cO_{K_5}$-valued point of $M_{\bK}$ compatible with $x$, 
and the degree of $K_5/K_1$ divides $C_1$.

(In the diagram below $\bullet$'s are the fibered products.)
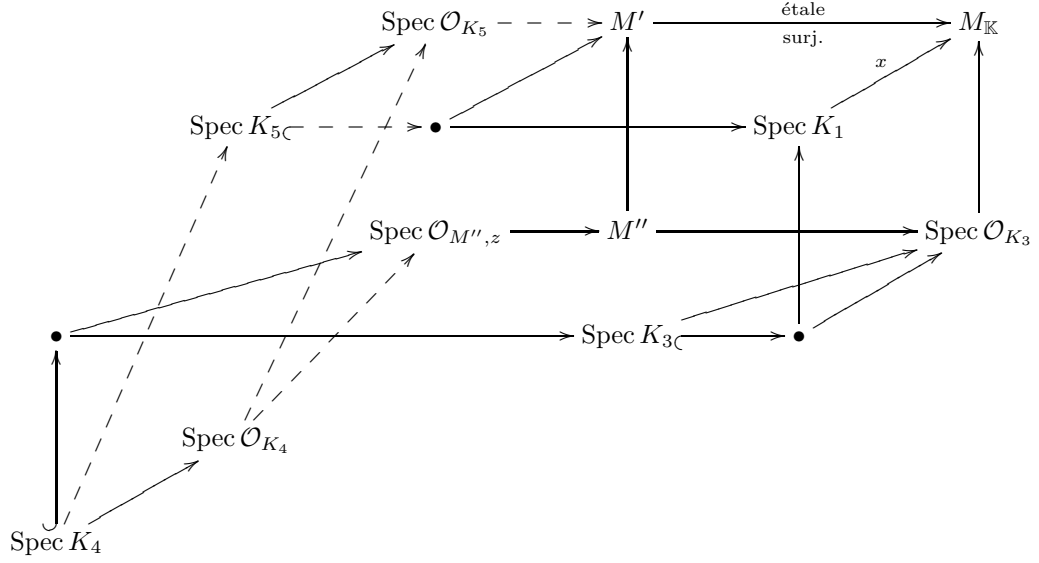
\[ \small 
\xymatrix{
&& \Spec \cO_{K_5} \ar@{-->}[r] & M' \ar[rr]^{\text{\'etale}}_{\text{surj.}} && M_{\bK} \\
&  \Spec K_5 \ar@{_{(}-->}[r] \ar[ru] & \bullet \ar[rr] \ar[ur] && \Spec K_1 \ar[ru]^x & \\
&& \Spec \cO_{M'',z} \ar[r] & M'' \ar[uu] \ar[rr] && \Spec \cO_{K_3} \ar[uu] \\
\bullet \ar[urr] \ar[rrr] &&& \Spec K_3 \ar@{_{(}->}[r] \ar[rru] & \bullet \ar[uu] \ar[ru] & \\
& \Spec \cO_{K_4} \ar@{-->}[uur] \ar@{-->}[uuuur] &&&& \\
\Spec K_4 \ar@{^{(}->}[uu] \ar[ur] \ar@{-->}[uuuur] & &&&&
}
\]
\end{proof}

\begin{rem}
If $M_\bK = M_{2d, \bK(3), \bZ_{(p)}}$ is already a scheme (not merely an algebraic space),
then it follows immediately that $M_\bK$ admits a $\Spec \cO_{K_1}$-valued point, 
we obtain an explicit bound $C = C_2 = \lvert \GL(21, \bF_3) \rvert$ independent of $d$, 
and moreover we can take the field extension to be unramified.
But we do not know whether $M_\bK$ is a scheme.
\end{rem}

\section{Some counterexamples} \label{sec:counterexample}

In this section 
we construct two explicit examples.

In Section \ref{subsec:example-algsp} we construct a K3 surface 
which has good reduction with an algebraic space model but not with a scheme model.
In Section \ref{subsec:example-unram} we construct a K3 surface
which has good reduction with a scheme model only after 
a base field extension.

\subsection{A sufficient condition for bad reduction}

We introduce a sufficient condition for a K3 surface not to have good reduction with a scheme model.

\begin{prop} \label{prop:badreduction}
Let $\cX$ be an irreducible projective scheme over $\cO_K$.
Assume that the generic fiber $X_K$ is a K3 surface 
and that the special fiber has at least one rational double point and has no other singularities.
If $X_K$ has Picard number one,
then $X_K$ does not have good reduction with a scheme model.
If furthermore $X_K$ has geometric Picard number one (i.e., $X_{\overline K}$ has Picard number one), 
then $X_K$ does not have potential good reduction with a scheme model.
\end{prop}

\begin{proof}
We prove the former assertion (then the latter follows immediately). 

Assume to the contrary that there exists a \emph{scheme} $\cX'$ 
which is a proper smooth model of $X_K$ (thus achieving a good reduction).

Let $D$ be an effective divisor which generates $\Pic(X_K)$, 
$\cD'$ its closure (as Weil divisor) in $\cX'$,
and $D'_k$ the restriction of $\cD'$ to $X'_k$.
Then either $D'_k$ is ample or not. 
We show that each assumption leads to a contradiction.

First we assume that $D'_k$ is not ample.
By the Nakai--Moishezon criterion, 
this means 
either $D_k'^2 \leq 0$ or $D'_k \cdot C \leq 0$ for some curve $C \subset X'_k$.
Since $D_k'^2 = D'^2 > 0$, the former cannot occur. 
So take an irreducible curve $C \subset X'_k$ with $D'_k \cdot C \leq 0$.
(The following argument is essentially given in \cite{Artin:brieskorn}*{first page}.)
Take an affine open subset $U \subset \cX'$ with $C \cap U \neq \emptyset$.
Since $\cX'$ is regular, the complement $\cX' \setminus U$ is a divisor, 
and its components $\cZ_i$ are all linearly equivalent to some positive multiple of $\cD'$ 
(since $X$ has Picard number one and the special fiber is integral).
We have $\cZ_i \cdot C \leq 0$ by linear equivalence, and $\cZ_i \cdot C \geq 0$ since $\cZ_i$ does not contain $C$.
Hence we have $(\cX' \setminus U) \cdot C = 0$.
Then the affine scheme $U$ contains a complete curve $C$, which is absurd.

Now we assume that $D'_k$ is ample.
By \cite{EGA3-1}*{Th\'eor\`eme 4.7.1} $\cD'$ is a relatively ample divisor of $\cX$ over $\cO_K$.
Then $(\cX', \cD')$ and $(\cX, \cD)$ 
(where $\cD$ is the closure of $D$ in $\cX$)
have isomorphic generic fibers (as polarized varieties)
but non-isomorphic special fibers (as varieties),
as $X'_k$ is smooth and $X_k$ is singular.
The following theorem shows that this is impossible.
\end{proof}

\begin{thm}[cf.\ Matsusaka--Mumford \cite{Matsusaka--Mumford}*{Theorem 2}] \label{thm:MMmodified}
Let $\cX_1$ and $\cX_2$ be irreducible schemes proper over $\cO_K$, 
$\cL_1$ and $\cL_2$ ample invertible sheaves on $\cX_1$ and $\cX_2$ respectively, 
and $\cT \subset \cX_1 \times_{\cO_K} \cX_2$ an irreducible closed subscheme, flat over $\cO_K$,  
such that its generic fiber $T_K$ gives an isomorphism of polarized varieties $(X_{1 K}, L_{1 K})$ and  $(X_{2 K}, L_{2 K})$. 
Assume $\cX_1$ is smooth over $\cO_K$, 
and $\cX_2$ has smooth generic fiber and normal irreducible special fiber.
Assume further that $X_{2 k}$ is not ruled.
Then an irreducible component of $T_k$ gives an isomorphism between $X_{1 k}$ and $X_{2 k}$.

\end{thm}
If we further assume that $\cX_2$ is smooth over $\cO_K$, this is \cite{Matsusaka--Mumford}*{Theorem 2}.
However their proof works under this weaker assumption, as we check later in Section \ref{sec:MM}.

We also need the following proposition
(this is merely a restatement of the argument on Artin's resolution in Section \ref{sec:proof}).

\begin{prop} \label{prop:algspresolution}
Let $\cX$ be an irreducible proper scheme over $\cO_K$.
Assume that the generic fiber $X_K$ is a smooth surface
and that the special fiber is smooth except for (finitely many) rational double points.
Then
the surface $X_K$ has potential good reduction with algebraic space model. 
\end{prop}

\subsection{Example: good reduction only with an algebraic space model} \label{subsec:example-algsp}

We construct an explicit example of a K3 surface 
which has potential good reduction with algebraic space models
but not with scheme models.
As we mentioned in the introduction,
this indicates that allowing algebraic spaces is essential when considering reduction of K3 surfaces, 
in contrast to the case of abelian varieties.

Let $p \geq 7$ be a prime and $K = \bQ_p$ (hence $\cO_K = \bZ_p$).
Let $f \in \bZ[x,y,z]$ be a 
homogeneous sextic polynomial satisfying the congruences
\begin{align*}
 f &\equiv 
  2 x^6 + x^4 y^2 + 2 x^3 y^2 z + x^2 y^2 z^2 + x^2 y z^3 + 2 x^2 z^4 \\
   & \qquad
    + x y^4 z + x y^3 z^2 + x y^2 z^3 + 2 x z^5 + 2 y^6 + y^4 z^2 + y^3 z^3 
   \pmod 3, \\
 f &\equiv 
  y^6 + x^4 y^2 + 3 x^2 y^4 + 2 x^5 z + 3 x z^5 + z^6 
   \pmod 5, \;\text{and} \\
 f &\equiv 
  x y z^4 + x^6 + y^6
   \pmod p.
\end{align*}
(Such an $f$ clearly exists from the Chinese remainder theorem.)
Let $\cX$ be the double covering of $\bP^2_{\bZ[1/2]}$ 
defined by the equation $w^2 = f(x,y,z)$ (so that it ramifies at the sextic defined by $f = 0$).
As Elsenhans--Jahnel \cite{Elsenhans--Jahnel}*{Example 5.1.1} showed, 
the  congruences modulo $3$ and $5$ imply 
 that $X_\bQ$ (and hence $X_K$) 
are smooth K3 surfaces and have geometric Picard number one.
On the other hand, from the congruence modulo $p$, 
the special fiber $X_{\bF_p}$ of $\cX_{\cO_K}$ has exactly one rational double point at $x=y=0$.

It follows from Propositions \ref{prop:badreduction} and \ref{prop:algspresolution} that $X_K$ is a desired example.

\subsection{Example: good reduction only after unramified field extension} \label{subsec:example-unram}

We construct a K3 surface over $K$
which has good reduction with a scheme model over the integer ring of some unramified extension of $K$ but not of $K$ itself.
(But there remains the possibility that it has good reduction with an algebraic space model over $\cO_K$ without extension.)
\footnote{We showed in a subsequent paper \cite{Liedtke--Matsumoto}*{Theorem 6.2}
that the example below does not have good reduction over $\cO_K$ even in the category of algebraic spaces.}

This shows another difference between K3 surfaces and abelian varieties, 
as this situation does not occur in the case of abelian variety:
whether an abelian variety has good reduction or not can be completely determined by the action of the inertia group, 
and an unramified extension does not change the inertia group.

The idea of our example is as follows:
the model $\cX$ over $\cO_K$ has isolated singularities in the special fiber, 
which can be resolved by blowing up exactly one of the two curves $\cC_+$ and $\cC_-$, 
but these curves (and the corresponding resolutions and models) are defined only after base change by $K'/K$.

Let $\phi, \Phi \in \bZ[x,y,z,w]$ be the (homogeneous) polynomials
\begin{align*}
 \phi &=
  x^3 - x^2y - x^2z + x^2w - x y^2 - x yz + 2x yw + x z^2 + 2x zw \\
   & \qquad + y^3+ y^2z - y^2w + yz^2 + yzw - yw^2 + z^2w + zw^2 + 2w^3 ,
  \;\text{and}\\
 \Phi &= w \phi + (z^2+xy+yz)(z^2+xy).
\end{align*}
Van Luijk \cite{vanLuijk:K3one}*{proof of Theorem 3.1} proved that 
the K3 surface $Y_2 = (\Phi = 0) \subset \bP^3_{\bF_2}$
has geometric Picard number two (by computing its zeta function), 
and that any quartic surface defined over $\bQ$ by a polynomial congruent to $\Phi$ modulo $2$
is a smooth K3 surface and has geometric Picard number at most two (by using a reduction argument).

Now let $p \geq 5$ be a prime, 
$a$ an integer with $a \not\equiv 0, \frac{27}{16} \pmod p$,
and $b = cp^2$ where $c$ is an integer which is a quadratic nonresidue modulo $p$ and satisfies $c \equiv 1 \pmod 8$.
Let $f \in \bZ[x,y,z,w]$ be a homogeneous cubic polynomial satisfying the congruences
\begin{align*}
 f &\equiv
  \phi \pmod 2,  \;\text{and} \\
 f &\equiv
  x^3+y^3+z^3+aw^3
  \pmod p.
\end{align*}
Let $\cX \subset \bP^3_\bZ = \Proj \bZ[x,y,z,w]$ be the ``quartic surface''
defined by the equation $F = wf + g^2 - b h^2 = 0$,
where 
\[
 g = p z^2 + xy + \frac{p}{2} yz, \quad \text{and} \quad 
 h = \frac12 yz
\]
(note that $g^2 - b h^2$ has integral coefficients although $g$ and $h$ do not).

Let $K = \bQ_p$.
We show that $X_K$ is a desired example.

Since $F \equiv \Phi \pmod 2$,
$X_K$ is a smooth K3 surface and
has geometric Picard number at most two from the above argument.

Over the (unramified) quadratic extension $K' = K(\sqrt{b})$ of $K$, 
we have two distinct irreducible divisors $C_+$ and $C_-$, 
where 
$C_{\pm}$ are respectively defined by $w = g \pm \sqrt{b}\, h = 0$.
Since $C_+^2 < 0$ (since $C_+$ is a smooth rational curve)
but $C_+ \cdot C_- \geq 0$,
the classes of $C_+$ and $C_-$ in the Picard group are linearly independent.
Hence $X_{K'}$ has Picard number two.
Since the action of $\Gal(K'/K)$ on $\Pic(X_{K'})$ is nontrivial
(the nontrivial element takes $C_+$ to $C_-$), 
$X_K$ has Picard number strictly less than two, hence one.

One easily checks that the singular points of 
$X_{\bF_p} = (wf + g^2 = 0) \subset \bP^3_{\bF_p}$ 
are exactly the six points where $w = f = g = 0$
(recall that $f \equiv x^3 + y^3 + z^3 + aw^3$ and $g \equiv xy \pmod p$), 
and that these points are rational double points (of type $A_1$).
We conclude by using Proposition \ref{prop:badreduction}(2) that $X_K$ does not have good reduction.

On the other hand $X_{K'}$ has good reduction.
We construct two smooth proper models $\cX'_+$ and $\cX'_-$.
Let $\cC_+$ and $\cC_-$ be the subschemes of $\cX_{\cO_{K'}}$ defined by $w = g + \sqrt{b}\, h = 0$ and $w = g - \sqrt{b}\, h = 0$ respectively.
Let $\psi_+ \colon \cX'_+ \rightarrow \cX_{\cO_{K'}}$ be the blow-up at $\cC_+ \subset \cX_{\cO_{K'}}$.
Since the divisor $\cC_+$ is Cartier at all the points where at least one of $w$, $f$, $g + \sqrt{b}\, h$, and $g - \sqrt{b}\, h$ is nonzero,
the morphism $\psi_+$ is an isomorphism outside the subscheme $(w=f=g=bh=0)$.
One can easily check that this locus $(w=f=g=bh=0)$
is the set of the six singular points of the special fiber.
Some computation on local equations shows that $\psi_+$ 
is the minimal desingularization on the special fiber
(and is an isomorphism on the generic fiber).
Therefore $\cX'_+$ is a  smooth proper model of $X_{K'}$.
By blowing up $\cC_-$ we similarly obtain $\cX'_-$, which is another smooth proper model.

\subsection{The Matsusaka--Mumford theorem} \label{sec:MM}

We prove Theorem \ref{thm:MMmodified}.
This is a small refinement of the original theorem of Matsusaka--Mumford.

First we show the following theorem.

\begin{thm}[cf. Matsusaka--Mumford \cite{Matsusaka--Mumford}*{Theorem 1}] \label{thm:MM1}
Let $\cX$, $\cY$, and $\cT \subset \cX \times_{\cO_K} \cY$ be irreducible schemes flat over $\cO_K$ 
(of the same relative dimension $n$).
Assume that the fibers of $\cX$ and $\cY$ are integral, 
 that $\cX$ is proper smooth over $\cO_K$, $\cY$ is normal, 
 and that $T_K$ gives a birational correspondence from $X_K$ to $Y_K$.
Assume moreover that $Y_k$ is not ruled.
Then the $n$-cycle $T_k$ admits a (unique) decomposition $T_k = T'' + T^*$ to effective $n$-cycles, 
with $T''$ a birational correspondence from $X_k$ to $Y_k$ 
and with $T^*$ satisfying $\pr_{1*}(T^*) = 0$ and $\pr_{2*}(T^*) = 0$ (where $\pr_i$ are the projections to $X_k$ and $Y_k$).
\end{thm}
\begin{proof}
By the compatibility of specialization and push-forward we have $\pr_{2 *}(T_k) = Y_k$, 
hence some (unique) component $T''$ of $T_k$ with multiplicity one satisfies $\pr_2(T'') = Y_k$.
We shall show $\pr_1(T'') = X_k$.

Let $A = \pr_1(T'') \subset X_k$.
It follows from \cite{Abhyankar:valuations}*{Proposition 3} that $T''$ is birationally equivalent to $A \times \bP^n$ for some $n \geq 0$.
(This follows from properties of the morphism $\cO_{\cX,\alpha} \to \cO_{\cT, \tau''}$ of local rings,
where $\alpha$ and $\tau''$ are respectively the generic points of $A$ and $T''$. 
We need both local rings to be regular, hence we assumed $\cX$ smooth and $\cY$ normal.)
Since $T''$ is birationally equivalent to $Y_k$ and $Y_k$ is not ruled, we have $n = 0$, hence $\pr_1(T'') = X_k$.
\end{proof}

\begin{proof}[Proof of Theorem \ref{thm:MMmodified}]
Take the decomposition $T_k = T'' + T^*$ of Theorem \ref{thm:MM1}.
Matsusaka--Mumford shows (under assumption that $\cX$ and $\cY$ are smooth) 
that the birational map induced by $T''$ is in fact an isomorphism.
Koll\'ar pointed out \cite{Kollar:toward}*{proof of Proposition 3.1.2},
that their argument works for normal (not necessarily smooth) varieties.
\end{proof}

\section{Application: potential good reduction of K3 surfaces with complex multiplications} \label{sec:CMK3}

As mentioned in \cite{Rizov:CMK3}*{Section 3.10 (A)}, 
the N\'eron--Ogg--Shafarevich type (potential) good reduction criterion 
implies the potential good reduction of K3 surfaces with complex multiplications (CM).

We recall the Hodge group and the Hodge endomorphism algebra of a complex K3 surface (see \cite{Zarhin:HodgegroupsK3} for details).
Let $X$ be a complex K3 surface and $V = \NS(X)^\perp \subset H^2(X, \bQ)$ the transcendental lattice.
The Hodge group $\Hdg$ of $X$ is the minimal algebraic subgroup of $\GL(V)$ defined over $\bQ$
such that $h(U^1) \subset \Hdg_\bR$, 
where $h \colon \bS = \Res_{\bC/\bR} \bG_m \to \GL(H^2(X, \bZ) \otimes \bR)$ is the morphism attached to the Hodge structure
and $U^1 \subset \bS(\bR) = \bC^*$ is the set of complex numbers of absolute value $1$.
It is known that the Hodge endomorphism algebra $E = \End_{\Hdg}(V)$ is either a totally-real field or a CM field,
and that $\Hdg = \SO(V, \psi)$ or $\Hdg = \U(V, \psi)$ respectively,
where $\psi : V \times V \to E$ is defined by
\[
 \langle ex, y \rangle = \trace_{E / \bQ}( e \psi(x,y)),
\]
where $\langle {}, {} \rangle$ is the pairing on the middle degree cohomology.

\begin{defn} \label{def:CMK3}
A complex K3 surface $X$ is said to have complex multiplication (CM)
if $\Hdg$ is commutative.
Equivalently, $X$ has complex multiplication if $E$ is a CM field and $\rank_E V = 1$.
\end{defn}

\begin{rem}
The definition of CM given by Rizov \cite{Rizov:CMK3}*{Definition 1.7}
is a weaker one that $E$ is a CM field (without any condition on $\rank_E V$).
He states \cite{Rizov:CMK3}*{Corollary 3.19} 
that if $X$ has complex multiplication (in this weaker sense) with Hodge endomorphism algebra $E$ 
then $X$ is defined over an abelian extension of $E$,
but this is correct only in our stronger sense (i.e. $\rank_E V = 1$).
\footnote{This correction is added after publication. 
I thank Lenny Taelman for informing this point to me.}
\end{rem}

\begin{thm} \label{thm:CMpotgoodred}
Let $X$ be a CM K3 surface (in the sense of Definition \ref{def:CMK3}), which is defined over some number field, say $K$.
Let $\fp$ be any prime of $K$ such that its residue characteristic $p$ satisfies $p > L^2 + 4$
for some ample line bundle $L$ on $X$.
Then $X$ has potential good reduction at $\fp$ (with an algebraic space model).
\end{thm}
\begin{rem}
As in our main theorem,
the assumption $p > L^2 + 4$ can be weakened to $p \geq 5$ if we admit the semistable reduction conjecture.
\end{rem}
This can be viewed as an analogue of the theorem that
abelian varieties with complex multiplication have good reduction \cite{Serre--Tate}*{Theorem 6}, 
although our result is restricted to large $p$.

This is also a generalization (for large $p$) of a previous result of ours on exceptional K3 surfaces.
Exceptional K3 surfaces are the K3 surfaces in characteristic $0$ of Picard number $20$ (which is the maximum possible value in characteristic $0$)
and they are examples of CM K3 surfaces.
We showed (\cite{Matsumoto:SIP}*{Corollary 0.5}) that exceptional K3 surfaces have potential good reduction for $p \neq 2,3$.

\begin{proof}[Proof of Theorem \ref{thm:CMpotgoodred}]
We may assume that $K$ contains $E$.
Take a prime $l \neq p$.
Since the Galois action preserves $\psi$,
the image of $G_K$ in $\GL(V \otimes \bQ_l)$ is contained in $\U(V_{\bQ_l}, \psi)(\bQ_l)$.
Since $\rank_E V = 1$, we have 
$\U(V_{\bQ_l}, \psi)(\bQ_l) \cong \Ker (\norm \colon (E \otimes \bQ_l)^* \to (E_0 \otimes \bQ_l)^*) \subset (E \otimes \bQ_l)^*$.
This group is abelian and contains a pro-$l$ group of finite index.

Since $\Hdg_{\bQ_l}$ is abelian, 
the action of $G_{K_\fp}$ factors through $G_{K_\fp}^\abel$.
Under the reciprocity map $K_\fp^* \to G_{K_\fp}^\abel$ of the local class field theory,
the image of the inertia group $I_{K_\fp}$ in $G_{K_\fp}^\abel$ is isomorphic to $\cO_{K_\fp}^*$,
a group which contains a pro-$p$-group of finite index.

Since $l \neq p$, it follows that the image of $I_{K_\fp}$ in $\U(V_{\bQ_l}, \psi)(\bQ_l)$ is finite.
Therefore, over a finite extension of $K$, $H^2_\et(X_{\overline {K_\fp}}, \bQ_l)$ is unramified,
and by Theorem \ref{thm:maintheorem} we conclude that $X$ has potential good reduction at $\fp$.
\end{proof}

\end{document}